\documentclass[letterpaper]{amsart}
\pdfoutput=1

\usepackage[T1]{fontenc}        % modern font encoding
\usepackage[utf8]{inputenc}     % ensure UTF-8 input (redundant for Lua/XeLaTeX)
\usepackage[english]{babel}     % correct hyphenation etc.
\usepackage{lmodern}            % improved Latin Modern fonts
\usepackage{wrapfig}
\usepackage{multicol}
\usepackage{multirow}
\usepackage{listings}
\usepackage{capt-of}

\usepackage{amsmath,amssymb,amsthm,mathtools,bm}
\usepackage[all]{xy}

\usepackage[margin=1in]{geometry}
\usepackage{xcolor,graphicx,xspace}
\usepackage{booktabs}

\usepackage[
  colorlinks=true,
  linkcolor=blue!60!black,
  citecolor=green!40!black,
  urlcolor=blue!60!black
]{hyperref}

\usepackage[nameinlink,capitalize]{cleveref}
\crefname{figure}{Figure}{Figures}

\usepackage[style=alphabetic, backend=bibtex]{biblatex}
\definecolor{bg}{rgb}{0.95,0.95,0.95}

\newcommand\OSCAR{\texttt{OSCAR}\xspace}

\newcommand{\OO}{\mathcal{O}}

\newcommand\QQ{\mathbb{Q}}
\newcommand\RR{\mathbb{R}}
\newcommand\CC{\mathbb{C}}
\newcommand\ZZ{\mathbb{Z}}
\newcommand\NN{\mathbb{N}}
\newcommand\PP{\mathbb{P}}
\newcommand\rank{\mathrm{rank}}

\newcommand\End{\mathrm{End}}

\newcommand\Spec{\mathrm{Spec}\,}
\newcommand\Hom{\mathrm{Hom}}
\newcommand\Ext{\mathrm{Ext}}
\newcommand\Cl{\mathrm{Cl}}

\theoremstyle{definition}
\newtheorem{theorem}{Theorem}[section]

\newtheorem{proposition}[theorem]{Proposition}
\newtheorem{lemma}[theorem]{Lemma}

\newtheorem{remark}[theorem]{Remark}
\newtheorem{example}[theorem]{Example}

\hypersetup{
  pdftitle={Computing \texorpdfstring{$A$}{A}-resultants via direct images}
  pdfauthor={Friedemann Groh, Matthias Zach},
  pdfsubject={Computational algebra and geometry}
}

\title{\Large Computing $A$-resultants via direct images}
\date{}

\author{Friedemann Groh}
\address{F. Groh: ISG Industrielle Steuerungstechnik GmbH, 
STEP, Gropiusplatz 10, 
70563 Stuttgart, Germany}

\author{Matthias Zach}
\address{
  M. Zach: 
  RPTU Kaiserslautern-Landau,
  Gottlieb-Daimler-Stra{\ss}e, Geb. 48,
  67663 Kaisers\-lautern,
  Germany 
}

\begin{document}

\maketitle

\begin{abstract}
  We improve a previously known theoretic method to compute $A$-resultants for suitable 
  monomial support 
  sets due to Weyman to the extent that it becomes computationally feasible 
  and effective. 
  This is achieved by introducing a new algorithm for the computation of direct images 
  of complexes of coherent sheaves on toric varieties. 
  The procedure does not rely on Gr\"obner basis computations at 
  any stage. 
\end{abstract}

\tableofcontents
\noindent

\section{Results}

For a finite set $A \subset \ZZ^m$ we denote the set of Laurent polynomials 
$f = \sum_{\nu \in A} a_\nu z^\nu$ with support in $A$ by $\CC^A$. 
Let $A_1,\dots, A_n \subset \ZZ^m$ be $n$ such sets and consider the system 
\begin{eqnarray}
  \label{eqn:PolynomialSystem}
  F &=& \left\{
    \begin{array}{ccc}
      f_1 &=&\sum_{\nu \in A_1} a_{1, \nu} \cdot z_0^{\nu_0} \cdot z_1^{\nu_1} \cdots z_m^{\nu_m} = 0\\
      & \vdots & \\
      f_n &=&\sum_{\nu \in A_n} a_{n, \nu} \cdot z_0^{\nu_0} \cdot z_1^{\nu_1} \cdots z_m^{\nu_m} = 0.
    \end{array}
    \right.
\end{eqnarray}
We write
\begin{equation}
  \label{eqn:DefinitionDiscriminant}
  \nabla_{A_1,\dots,A_n} = 
  \overline{
    \left\{
      (f_1, \dots, f_n) \in \CC^{A_1} \times \dots \times \CC^{A_n} : 
      f_1(z) = \dots = f_n(z) = 0 \textnormal{ for some } z \in (\CC^*)^{m}
    \right\}
  }
\end{equation}
for the Zariski closure of the set of parameters $(a_{j, \nu})_{j, \nu}$ 
with a non-empty solution set. 
In the case $n = m+1$ and under the additional assumptions that 
\begin{itemize}
  \item every $A_j$ generates $\RR^m$ and 
  \item all the $A_j$ together generate $\ZZ^m$
\end{itemize}
it is known that the discriminant is an irreducible, reduced hypersurface; 
cf. \cite[Chapter 8.1, Proposition 1.1]{GelfandKapranovZelevinsky94}.
Its defining polynomial 
\[
  \Delta_{A_1, \dots, A_{m+1}} \in \ZZ[a_{i, \nu} : i = 1, \dots, m+1, \, \nu \in A_i]
\]
is defined over the integers and unique up to sign. It is called the 
$A$-resultant for the support sets $A_1, \dots, A_{m+1}$. We discuss 
subsequent generalizations of the setup in Section \ref{sec:BroadeningScope}. 

\medskip
The achievement of this paper is to render a previously described 
theoretic approach 
due to Weyman \cite{Weyman94} 
for the computation of $A$-resultants algorithmic via 
the introduction of the \textit{Weyman functor} and its practical 
implementation. 

\medskip
The straightforward method to compute the $A$-resultant (or more generally 
the \textit{eliminant}, cf. Section \ref{sec:BroadeningScope} below) 
is to use an elimination 
ordering on the monomials of the polynomial ring 
\[
  \QQ[x_1,\dots, x_m, a_{i, \nu} : i = 1, \dots, m+1, \, \nu \in A_i]
\]
and compute a Gr\"obner basis 
which eliminates the variables $x_1,\dots, x_m$ from the ideal generated 
by the system (\ref{eqn:PolynomialSystem}). In practice, however, this fails 
already for reasonably sized examples due to the complexity of Buchberger's algorithm 
which is double exponential in the number of variables. What this means in practice 
can for instance be observed in the examples discussed in Sections \ref{sec:SturmfelsExample} 
and \ref{sec:ParametricImages}. 

To explain our angle of attack in a bit more detail: 
Gelfand, Kapranov, and Zelevinsky discuss 
a method to compute the $A$-resultant as the \textit{determinant of a complex of global sections} 
for suitable stable twists 
\cite[Chapter 2.2]{GelfandKapranovZelevinsky94}, 
as the \textit{determinant of a spectral sequence} 
\cite[Chapter 2.5]{GelfandKapranovZelevinsky94},
or from the \textit{Weyman complex} \cite{GelfandKapranovZelevinsky94}, 
\cite{Weyman94}. 
Picking up on Weyman's the ideas, we develop a general method 
for the computation of direct images 
of complexes of coherent sheaves on toric varieties via the \textit{Weyman functor}. 
This leads to a concrete algorithm with an efficient implementation to compute $A$-resultants 
via Equation (\ref{eqn:DeterminantOfWeymanComplex}) below. 
It allows for computations on a whole different 
level compared to elimination and 
our algorithm does not rely on Gr\"obner basis computations at any stage. 

As a natural byproduct we obtain a description of the variety 
$\nabla_{A_1,\dots, A_{m+1}}$ as the \textit{degeneracy locus} of a complex 
of free modules which often breaks down, at least locally, to the vanishing 
of a \textit{single} determinant of a square matrix. Compared with the method 
to compute the resultant as the determinant of a stably twisted complex of 
global sections 
(cf. \cite[Chapter 3, Theorem 4.2]{GelfandKapranovZelevinsky94}) this has 
the advantage that the complexes tend to be smaller. Moreover, in many 
cases of practical interest, it is not possible to arrive at a single 
matrix via \textit{stable} twisting while other twists in conjunction 
with our machinery lead to more favorable complexes; cf. Sturmfels' 
Example in Section \ref{sec:SturmfelsExample}.

\section*{Acknowledgments}

We would like to thank Martin Bies and Mi\c{k}elis E. Mi\c{k}elsons 
for numerous discussions on cohomology of toric line bundles and providing 
us with a preliminary implementation of $e_{\min}$ 
(see Equation (\ref{eqn:OptimalKFunction})). Further, we would like 
to thank Michael Joswig for introducing the authors to one another 
and Gregory Smith and David Eisenbud for pointing out references and 
feedback. Parts of this work have been carried out during a visit of the 
second named author at the Instituto de Matem\'aticas, Unidad Cuernavaca 
of the Universidad Nacional Aut\'onoma de M\'exico. 
This work has been funded by the Deutsche Forschungsgemeinschaft 
(DFG, German Research Foundation) – Project-ID 286237555 – TRR 195. 
Lastly, we would like to thank ISG Industrielle Steuerungstechnik GmbH in Stuttgart. 
The authors have not used any artificial intelligence products or services to produce or to assist 
with the production of this article or the related implementations. 

\subsection{Methods from the Green Book}

The methods for the computation of $A$-resultants discussed in 
\cite{GelfandKapranovZelevinsky94} 
proceed via the construction of a certain Koszul complex of coherent sheaves 
on a toric variety. 

For the given family of support sets $A_j \subset \ZZ^m$, $j=1, \dots, n$, let 
$P_j \subset \RR^m$ be the convex hull of $A_j$ and $Q$ the Minkowski sum 
of $P_1,\dots, P_{n}$. Each one of 
the $P_j$ appears as a face of $Q$ and gives rise to a divisor $D_j$
on the toric variety $X$ associated to $Q$.
This construction comes with an injective map 
\[
  \mathrm{res}_j \colon \CC^{A_j} \hookrightarrow V_j := H^0(X, \OO(D_j))
\]
We let $V = V_1 \times \dots \times V_n$ and denote the individual line 
bundles by $\mathcal L_i = \OO(D_i)$.

%and 
%\[
%  \Phi \colon X \to \PP (\CC^{A_p})^\vee \times \dots \times \PP (\CC^{A_p})^\vee
%\]
%the resulting morphism of projective varieties. 
The \textit{incidence variety} of common solutions of the linear systems for $V$ is 
\[
  W = \overline{\left\{(x, f_1,\dots, f_{n}) \in (\CC^*)^m \times \PP(V) : f_1(x) = \dots = f_n(x) = 0\right\}} \subset X \times \PP(V)
\]
with $(\CC^*)^m \subset X$ the natural embedding of the complex torus into $X$.
Since for every point $x \in (\CC^*)^m$ the condition on the coefficients 
$\{a_{i, \nu}\}$ of the $f_i$ given by evaluation at $x$ are linear and independent 
of one another, the incidence variety has the structure of a projective bundle 
of relative dimension $\dim V - n - 1$ 
over the complex torus. In particular, $W$ is irreducible and 
its codimension in $X \times \PP(V)$ is $n+1$. 

Since $X$ is complete, 
the projection of $W$ to $\PP(V_1) \times \cdots \times \PP(V_{n})$ 
is closed and hence its image and the affine cone over it
\begin{equation}
  \nabla_{\mathcal L_1, \dots, \mathcal L_{n}} = 
  \left\{(f_1,\dots, f_{n}) \in V_1 \times \dots \times V_{n} : 
  \{f_1 = \dots = f_{n} = 0\} \neq \emptyset
  \right\}
  \label{eqn:DiscriminantForLineBundles}
\end{equation}
are algebraic sets. It is easy to see from the definition 
(\ref{eqn:DefinitionDiscriminant}) and the fact that the complex 
torus is dense in any toric variety that (\ref{eqn:DiscriminantForLineBundles}) agrees with the $A$-discriminant. 

%We will make the additional assumption that $p = m+1$ and 
%\begin{itemize}
%  \item all the line bundles $\mathcal L_i = \OO(D_i)$ are very ample.
%\end{itemize}
\begin{proposition}(\cite[Proposition 3.1, Chapter 3.3]{GelfandKapranovZelevinsky94})
  \label{prp:DiscriminantIsAHypersurfaceLineBundles}
  When $n = m+1$ and all line bundles $\mathcal L_i$ are very ample, 
  the variety $\nabla_{\mathcal L_1,\dots, \mathcal L_{m+1}}$ is a 
  multihomogeneous irreducible hypersurface.
\end{proposition}
\noindent
The defining equation $\Delta_{\mathcal L_1,\dots, \mathcal L_{m+1}}$ is referred to as 
the $(\mathcal L_1, \dots, \mathcal L_{m+1})$-resultant. The assumption 
on the $\mathcal L_i$ to be very ample is used in the proof of 
Proposition \ref{prp:DiscriminantIsAHypersurfaceLineBundles} to verify 
that the projection of $W$ to its image is birational, i.e. generically $1$-to-$1$.

\medskip
To derive a formula for the resultant, one considers the following Koszul 
complex on the product $X \times \PP(V)$
\begin{equation}
  \label{eqn:KoszulResolution}
  \mathcal K(\mathcal L_1, \dots, \mathcal L_n) : 
  \begin{xy}
    \xymatrix{
      0 \ar[r] &
      \bigwedge^n E^\vee \otimes \OO(-n) \ar[r]^-{s} &
      \bigwedge^{n-1} E^\vee \otimes \OO(-n+1) \ar[r]^-{s} &
      \cdots & \\ %\ar[r]^s &
      \cdots \ar[r]^-s &
      \bigwedge^{2} E^\vee \otimes \OO(-2) \ar[r]^-{s} &
      \bigwedge^{1} E^\vee \otimes \OO(-1) \ar[r]^-{s} &
      \OO \ar[r] & 
      0 \\
    }
  \end{xy}
\end{equation}
where $E = \mathcal L_1 \oplus \dots \oplus \mathcal L_n$ 
and the $i$-th exterior power in this complex is located in 
cohomological degree $-i$. The contraction 
morphism $s$ is defined as follows. Consider the natural maps 
\[
  V \otimes V^\vee \to H^0(X, E) \otimes H^0(\PP(V), \OO(1)) \to 
  H^0(X \times \PP(V), E \otimes \OO(1)).
\]
Then the image of the identity element in 
$\End(V) \cong V \otimes V^\vee$ provides the section $s$ in $E \otimes \OO(1)$, 
the dual of $E^\vee \otimes \OO(-1)$. 
Under the presumptions made in Proposition \ref{prp:DiscriminantIsAHypersurfaceLineBundles}, 
this complex is a free resolution of the structure sheaf 
$\OO_W$, seen as an $\OO_{X \times \PP(V)}$-module 
(cf. \cite[Chapter 1, Proposition 1.4, Chapter 2, Corollary 2.8]{GelfandKapranovZelevinsky94}).

It is discussed in \cite[Chapter 3, Theorem 4.2]{GelfandKapranovZelevinsky94}) that for $n = m + 1$ and 
a stably twisted version $\mathcal K(\mathcal L_1, \dots, \mathcal L_n)\otimes \mathcal M$ 
the resultant 
\begin{equation}
  \label{eqn:ResultantAsDeterminantStableTwist}
  \Delta_{\mathcal L_1, \dots, \mathcal L_n} = \det \left(
  H^0(X \times \PP(V), \mathcal K(\mathcal L_1, \dots, \mathcal L_n) \otimes \mathcal M)
  \right)
\end{equation}
is equal to the determinant of the complex of global sections.
If one wishes to avoid the stable twisting by $\mathcal M$ (which usually leads to a significant 
increase in dimension of the vector spaces involved), it is suggested 
in \cite[Chapter 2.5 and Chapter 3.4 C]{GelfandKapranovZelevinsky94} 
to pass to the determinant of the spectral sequence 
\begin{equation}
  \label{eqn:TheSpectralSequence}
  E_1^{p, q} = R^q\pi_* \left(\bigwedge^{-p} E^\vee \otimes \OO(-p) \otimes \mathcal M \right)
  \Longrightarrow 
  R^{p+q}\pi_* \left(\mathcal K(\mathcal L_1, \dots \mathcal L_n) \otimes \mathcal M\right)
\end{equation}
where $\pi \colon X\times \PP(V) \to \PP(V)$ is the projection to the second factor. 

As spectral sequences are notoriously difficult to compute, the authors then 
propose in Section \cite[Chapter 2.5 E and Chapter 3.4 E]{GelfandKapranovZelevinsky94} 
to use the \textit{Weyman complex} with terms
\begin{equation}
  \label{eqn:WeymanComplex}
  W^i = \bigoplus_{p + q = i} R^q \pi_* \left(\mathcal K^p(\mathcal L_1, \dots, \mathcal L_n) \otimes \mathcal M\right).
\end{equation}
The differentials of the Weyman complex, as discussed in \cite{GelfandKapranovZelevinsky94} 
and \cite{Weyman94}, rely on the non-constructive \textit{choice} of a certain map $\alpha$
out of a homotopy class of morphisms 
(cf. \cite[Chapter 2.5, Lemma 5.7]{GelfandKapranovZelevinsky94}).
The maps of the Weyman complex have been made explicit in some particular 
cases, e.g. for $X = \PP^n$ by D'Andrea and Dickenstein in \cite[Section 5]{DAndreaDickenstein01}; 
see Section \ref{sec:RelatedWork} for some more such examples. 
The method we propose here leads to \textit{natural} and \textit{algorithmic} 
construction of this $\alpha$ 
so that our Weyman complex becomes a somewhat canonical representative of 
$R\pi_* \left(\mathcal K(\mathcal L_1, \dots, \mathcal L_n) \otimes \mathcal M\right)$
on $\PP(V)$ up to quasi-isomorphism; cf. 
\cite[Chapter 2.5, Lemma 5.6]{GelfandKapranovZelevinsky94}. 

For $n = m + 1$ and monomial support sets $A_0, \dots, A_m$ satisfying 
the two assumptions made in the introduction, 
the discriminant can then be computed as the determinant of that complex, 
\begin{equation}
  \label{eqn:DeterminantOfWeymanComplex}
  \Delta_{A_0, \dots, A_m} = \Delta_{\mathcal L_0, \dots, \mathcal L_m} = 
  \det R\pi_* \left(\mathcal K(\mathcal L_0, \dots, \mathcal L_m)\otimes \mathcal M \right)
\end{equation}
\cite[Chapter 2, Theorem 5.5]{GelfandKapranovZelevinsky94}, but without the 
restriction that it would be known only up to multiplication with an undetermined 
non-zero constant. 

\subsection{Broadening the scope of applicability}
\label{sec:BroadeningScope}
The assumptions made on the support sets in the introduction have proven to be unnecessarily 
restrictive and have been weakened over the last years. We summarize some of these results 
to broaden the scope of applicability for our methods. We keep the notation 
from the previous section. 

\medskip
We will be interested exclusively in the case where $\nabla_{A_0,\dots,A_n}$ 
is a hypersurface so that $\dim W = \dim \nabla_{A_0, \dots, A_n}$. 
To distinguish this case from those with higher codimension, we can use 
the following combinatorial result of Sturmfels.

\begin{theorem}(\cite[Theorem 1.1]{Sturmfels93})
  Let $A_0, \dots, A_n \subset \ZZ^n$ be finite support sets. Then the codimension of 
  $\PP\nabla_{A_0,\dots,A_n} \subset \prod_{j=0}^{n} \PP(\CC^{A_j})$ equals 
  \[
    \max_{J \subset \{0, \dots, n\}} \left\{|J| - \rank 
    \left\{ \sum_{j \in J} \lambda_j \cdot v_j : v_j \in A_j, \, \lambda_j \in \ZZ, \, \sum_{j \in J} \lambda_j = 1 \right\}\right\}.
  \]
\end{theorem}

We will henceforth assume $\nabla_{A_0, \dots, A_n}$ to be of codimension one.

\medskip

For the definition of the $A$-resultant in this setting we follow 
\cite{DAndreaSombra15}. Since $\nabla = \nabla_{A_0, \dots, A_n} \subset V$
is an irreducible hypersurface, there exists a reduced defining equation 
$h \in \QQ[a_{i, \nu} : i = 0,\dots, n,\, \nu \in A_i]$ which is unique 
up to multiplication by a unit. This polynomial will be referred to as the 
\textit{eliminant} of the system. For the $A$-resultant we then consider 
the algebraic cycle given by 
\begin{equation}
  \label{eqn:DefinitionMultiplicityResultant}
  \pi_*([W]) = [K(W) : K(\nabla)] \cdot [\nabla]
\end{equation}
in the Chow group of $V$; see e.g. \cite{Fulton98}. By definition, the multiplicity 
$m = [K(W) : K(\nabla)]$
of this cycle along $\nabla$ is the transcendence degree of the respective function 
fields; a combinatorial formula for $m$, which is based directly on the support sets, was given in 
\cite[Proposition 3.13]{DAndreaSombra15}.
The $A$-resultant is then defined as any polynomial 
$\Delta_{A_0, \dots, A_n} \in \QQ[a_{i, \nu} : i = 0,\dots, n,\, \nu \in A_i]$ 
defining this cycle, 
e.g. $\Delta_{A_0, \dots, A_n} = h^m$. 

\medskip
We shall now provide a brief discussion of the following fact, 
as we are unaware of a concise exposition of this in the literature. 

\begin{proposition}
  \label{prp:BroaderApplicability}
  Over the complex numbers Equation (\ref{eqn:DeterminantOfWeymanComplex}) applies whenever 
  one has $n+1$ monomial support sets $A_0, \dots, A_n$ in $n$ variables and 
  $\nabla_{A_0, \dots, A_n}$ has codimension one.
\end{proposition}

\begin{proof}
The association of a toric variety $X$ and a set of base point 
free Cartier divisors $D_j$ to the support sets $A_j$ applies regardless of 
our particular assumptions; see \cite[Proposition 6.1.1 (b)]{CoxLittleSchenk10}.
Again, the Koszul complex (\ref{eqn:KoszulResolution}) 
provides a free resolution 
of the structure sheaf $\OO_W$ of the incidence variety. 
We have the following commutative diagram 
\[
  \begin{xy}
    \xymatrix{
      W \ar@{^{(}->}[r]^-\iota \ar[d]_{\overline \pi} &
      X \times V \ar[d]^\pi \\
      \nabla \ar@{^{(}->}[r]^-\kappa & V
    }
  \end{xy}
\]
where $\overline \pi$ denotes the restriction of $\pi$ to 
the incidence variety $W$ and its image. 

Since $\overline \pi$ is proper and generically finite, 
there exists a Zariski open 
subset $U \subset \nabla$ over which $\overline \pi$ is finite. 
Let $D \subset X$ be the complement of the complex torus $(\CC^*)^n \subset X$.
We subsequently shrink $U$ to contain no points which lay in the image 
of $(D \times V) \cap W$ or the image of the singular locus of $W$. 
Again, this set is non-empty and Zariski open. 
Considering $\nabla$ with its reduced structure, also the set of regular points 
is Zariski open and non-empty and we may replace $U$ by its intersection with 
the smooth locus of $\nabla$. Let $q \in U$ be a regular value of $\overline \pi$ 
restricted to the smooth locus of $W \cap ((\CC^*)^n \times V)$. Then locally around 
$q$ the map $\overline \pi$ is a finite covering morphism of smooth complex manifolds. 

We have multiple ways to describe the stalk at $q$ of the direct image of $\OO_W$
in the complex analytic category:
\begin{eqnarray*}
  \bigoplus_{i=1}^{m'}\OO_{V, q}/\langle h\rangle &\cong&
  \kappa_*\left(\bigoplus_{i=1}^{m'}\OO_{\nabla, q}\right)\cong
  \kappa_* \left(\bigoplus_{p \in \overline \pi^{-1}(\{q\})} \OO_{W, p}\right) \cong
  \left(R\kappa_* \left(R\overline \pi_* \OO_W\right)\right)_q \\
  &\cong& 
  \left(R(\kappa \circ \overline \pi)_* \OO_W\right)_q
  \cong
  \left(R(\pi \circ \iota)_* \OO_W\right)_q \\
  & \cong & 
  \left(R\pi_* \left(\iota_* \OO_W\right) \right)_q 
  \cong 
  \left(R\pi_* \mathcal K(\mathcal L_0, \dots, \mathcal L_n) \right)_q.
\end{eqnarray*}
In all these equalities $\cong$ denotes a quasi-isomorphism of complexes of stalks 
of coherent sheaves. Here we have used that $\iota$, $\kappa$, and $\overline \pi$ 
are finite and hence $R\kappa_* = \kappa_*$ for these maps. Furthermore, we need 
\cite[Proposition 2.3.3]{Dimca04} for $R(f \circ g)_* = Rf_* \circ Rg_*$ and 
we have expressed $\iota_*\OO_W$ as $\mathcal K(\mathcal L_0, \dots, \mathcal L_n)$ 
and $\kappa_*\OO_\nabla$ as $\OO_V/\langle h\rangle$ for the defining equation $h$ 
of $\nabla$ up to quasi-isomorphism. 

From this we can deduce several facts. First of all, 
$\left(\overline \pi_* \OO_W\right)_q$ is a free, finite $\OO_{\nabla, q}$-module 
of rank $m'$ equal to the number of points $p$ in the preimage of $q$. This 
also holds in the algebraic category and hence, localizing first at $q$ and 
subsequently at the generic point of $\nabla$, we see that 
$m' = [K(W) : K(\nabla)]$ equals the transcendence degree and thus the multiplicity 
of the $A$-resultant at $\nabla$. 

Second, we may compute the determinant of the right hand side over the ring 
$\OO_{V, q}$. According to \cite[Appendix A, Theorem 30]{GelfandKapranovZelevinsky94} 
the multiplicity of the determinant at $\nabla$ equals 
\[
  \mathrm{ord}_h\left(\det(R\pi_*\mathcal K(\mathcal L_0, \dots, \mathcal L_n)_q)\right) 
  = 
  \sum_i (-1)^i \mathrm{mult}_{\langle h\rangle}
  \left(R^i\pi_*\mathcal K(\mathcal L_0, \dots, \mathcal L_n)_q\right)
\]
where $\mathrm{mult}_{P} M$ is defined as the length of the localization $M_P$ of a module 
$M$ at a prime $P$ and $\mathrm{ord}_h f$ as the exponent of $h$ in a prime factorization 
of $f$. Turning towards the left hand side of the above chain of quasi-isomorphisms, 
we see that 
$\mathrm{mult}_h \left(
  \bigoplus_{i=1}^{m'}\OO_{V, q}/\langle h\rangle
  \right)
  = m'
$
and therefore 
\begin{equation}
  \label{eqn:MultiplicityOfDeterminant}
  \mathrm{ord}_h \left( \det \left(
  R^i\pi_*\mathcal K(\mathcal L_0, \dots, \mathcal L_n)_q\right)
  \right) = m'.
\end{equation}
In order to establish the same statement in the polynomial ring for 
$V$, we only have to observe that 
$\left(R\pi_*\mathcal K(\mathcal L_0, \dots, \mathcal L_n)\right)_{q'}$
is quasi-isomorphic to zero for points $q' \in V$ outside of $W$. 
Then the determinant $\delta$ of the polynomial direct image 
$R\pi_* \mathcal K(\mathcal L_0, \dots, \mathcal L_n)$ can not vanish at any of those 
points and the Nullstellensatz implies that the reduced equation $h$ 
for $\nabla$ is contained in the radical of $\langle \delta\rangle$.
Now the set $U \subset \nabla$ is dense and hence the minimal exponent $k$ 
such that $h^k \in \langle \delta\rangle$ has to coincide with $m'$. 

We leave the adaptation of these arguments to the general case with a non-zero 
twist by $\mathcal M$ to the reader.
\end{proof}

\subsection{The Weyman functor}

Let $X$ be a normal toric variety without torus factors and $\Spec R$ 
an affine algebraic scheme of finite type defined over $\mathbf k$. 
We will assume the latter to be either field or the ring of integers $\ZZ$. Denote the Cox 
ring of $X$ by $S$ and let $S_R := S \otimes_{\mathbf k} R$ be its relative version 
over $\Spec R$. 
Consider a right-bounded (co-)complex\footnote{In this paper we will 
only consider cocomplexes, but refer to them as complexes for brevity.} 
of finitely generated $S_R$-modules 
\begin{equation}
  \label{eqn:ComplexOfGradedModules}
  \begin{xy}
    \xymatrix{
      \cdots \ar[r] & 
      M^{p_0 - 2} \ar[r] & 
      M^{p_0 - 1} \ar[r] & 
      M^{p_0} \ar[r] & 
      0.
    }
  \end{xy}
\end{equation}
Such a complex sheafifies in a natural way to a complex of 
coherent sheaves $\widetilde M^i$ on $X \times \Spec R$. 
Let $\pi \colon X \times \Spec R \to \Spec R$ be the projection to the 
second factor. Then $R \pi_* \widetilde M^\bullet$ is a right-bounded complex of coherent 
sheaves on $\Spec R$, i.e. a complex of finitely generated $R$-modules. 

The \textit{Weyman functor} is a functor 
\[
  \mathbb W \colon C^b(\textnormal{free } S_R-\textnormal{mod}) \to 
  C^b(\textnormal{free } R-\textnormal{mod})
\]
from the category of right-bounded complexes of free, graded $S_R$-modules to the 
category of right-bounded complexes of free $R$-modules such that we have a 
commutative diagram of functors 
\begin{equation}
  \label{eqn:WeymanFunctorAsCompositionOfFunctorsDiagram}
  \begin{xy}
    \xymatrix{
      C^b(\textnormal{free } S_R-\textnormal{mod}) \ar[d]_{\widetilde \cdot} 
        \ar[r]^{\mathbb W} &
      C^b(\textnormal{free } R-\textnormal{mod}) \ar[d] \\
      D^b(\textnormal{Coh}(X \times \Spec R))\ar[r]^{R\pi_*} &
      D^b(R-\textnormal{mod})).
    }
  \end{xy}
\end{equation}
%\begin{equation}
%  \lambda \circ \mathbb W = R\pi_* \circ \widetilde{\cdot} \colon M^\bullet \mapsto 
%  \lambda (\mathbb W(M^\bullet)) \cong_{\textnormal{qis}} R\pi_* \widetilde M^\bullet
%  \label{eqn:WeymanFunctorAsCompositionOfFunctors}
%\end{equation}
%up to quasi-isomorphism where $\lambda \colon C^b(R-\textnormal{mod}) \to D^b(R-\textnormal{mod})$ is the natural map from the category of complexes into its derived category. 
Note that the requirement on the modules to be free is not really a restriction 
since up to homotopy 
we may replace any right-bounded complex of finitely generated $S_R$-modules 
by a graded free resolution thereof.
Given a complex of free graded $S_R$-modules $M^\bullet$ with 
terms, say, $M^p = \bigoplus_{j=1}^{r_p} S_R[-\alpha_{p, j}]$, 
the Weyman functor thus provides 
us with a natural representative for the direct image of the associated complex of 
coherent sheaves $\widetilde M^\bullet$ on along $\pi$ up to quasi-isomorphism.

The terms of the Weyman complex $\mathbb W(M^\bullet)$ are 
\begin{equation}
  \label{eqn:TermsOfWeymanComplex}
  \mathbb W^i(M^\bullet) = \bigoplus_{p + q = i} R^q \pi_* \widetilde M^p
  =\bigoplus_{p + q = i} \bigoplus_{j = 1}^{r_p} R^q \pi_* \OO(-\alpha_{p, j}), 
\end{equation}
i.e. the diagonals on the first page of the spectral sequence 
\[
  E_1^{p, q} = R^q \pi_* \widetilde M^p \Longrightarrow R^{p+q}\pi_* \widetilde M^\bullet.
\]
Note that the summands on the right hand side 
are necessarily free $R$-modules since, due to the 
product structure of $X \times_{\mathbf k} \Spec R$, we have 
\[
  R^q \pi_* \OO(-\alpha) \cong H^q(X, \OO(-\alpha)) \otimes_{\mathbf k} R
\]
for the finite dimensional vector spaces $H^q(X, \OO(-\alpha))$ over $\mathbf k$.

The construction of the morphisms 
$\phi^i \colon \mathbb W^i(M^\bullet) \to \mathbb W^{i+1}(M^\bullet)$ 
of the Weyman complex is more involved. The decomposition of the terms as 
direct sums endows $\phi^i$ with a block structure where 
\[
  \phi^{r}_{p, q} \colon R^{q} \pi_* \widetilde M^p \to 
  R^{q-r+1} \pi_* \widetilde M^{p+r}, \quad p+q=i
\]
and $\phi^{r}_{p, q} = 0$ for $r < 1$ and $r > \dim X$.
For $r = 1$ these maps coincide with those on the first page of the spectral 
sequence. For $r > 1$ the maps $\phi^{r}_{p, q}$ have their domain and codomain 
in the same positions as their analogs in the computation of the spectral sequence, 
but they are of a different nature. 
Either way, the representing matrix for $\phi^i$ will be of block lower triangular form. 
The details of the construction of these maps will be discussed in Section 
\ref{sec:DifferentialsInTheWeymanComplex}. 

\medskip
We have implemented the Weyman functor up to its value on morphisms of complexes 
in the computer algebra system \OSCAR \cite{OscarBook}, building on previous 
work \cite{Zach25} of the second named author. 
It should soon become available on the system's development branch.
All test runs in this article have been performed on a 64 bit architecture with 
an 12th Gen Intel(R) Core(TM) i5-1245U processor and 32 GB of RAM. 
We have worked with the current \OSCAR development branch on commit 
\verb|e48c5c5df6af70fea315afe6a2fe24a69961c95a| and julia 1.11.1 on a ubuntu desktop system.

\subsection{Synopsis to related work}
\label{sec:RelatedWork}

We believe that the construction of the Weyman functor as a means to 
compute direct images and sheaf cohomology should be of interest in its own 
right; probably even as a relevant competitor to recent work by Brown and 
Erman \cite{BrownErman21}. 

As for the $A$-resultants, our application of Weyman
functors continues the approach of \cite{GelfandKapranovZelevinsky94}, 
but it also falls into the scope of research going back to Canny and Emiris 
\cite{CannyEmiris93}, \cite{CannyEmiris95}, \cite{CannyEmiris00}, 
Sturmfels \cite{Sturmfels93}, and more 
recent contributions by Bus\'e, Checa, and Emiris 
\cite{ChecaEmiris24}, \cite{BuseCheca22}. These authors use a 
``hands-on'' constructive approach which does not refer to determinants of complexes and 
allows for more general support sets $A_i$ than those meeting the requirements in \cite{GelfandKapranovZelevinsky94}. 
However, their work can to a certain extent be interpreted within the geometric context 
used in this article. For instance, the shift of Newton-polytopes
by a vector with rational components introduced by Emiris and Canny can be understood as
a twist of the line bundle specified by the polytope with a reflexive bundle of rank one.
From our geometric point of view, they consider the differential of a complex of global sections 
in cohomological degree $-1$, where all higher sheaf cohomologies vanish as a result of the Canny-Emiris shift. 
Correspondingly, the
resultant is then recovered from the greatest common divisors of the differential's maximal minors; 
cf. \cite[Appendix A, Theorem 34]{GelfandKapranovZelevinsky94}. 

The approach, which is algorithmically realized by us, has the advantage that resultants can
be computed directly as determinant of a single complex or even one single matrix. Moreover, 
there are examples, in which the computation of determinants is the bottleneck 
(see e.g. Section \ref{sec:ParametricImages}), but not computing the direct image itself. 
In such cases it is certainly preferable to not be obliged to compute several minors 
and greatest common divisors, before arriving at the resultant. 
%However, it would be interesting to compare runtimes on benchmark examples. 
 
Cattani, Dickenstein, and Sturmfels \cite{Cattani1998} consider systems with Newton polytopes
that coincide up to scaling. They assume sufficiently generic, but fixed coefficients so that the system's
polynomials form a regular sequence in the Cox ring of the toric variety specified by the normal fan of the Newton 
polytopes. The resulting Koszul complex is twisted so that there is only one single higher sheaf cohomology
in the last term at the upper left of the first page of the spectral sequence (\ref{eqn:TheSpectralSequence}), 
which has rank one. In this special setting, the authors 
can construct the Weyman complex and the block for the differential, mapping this 
cohomology to the global sections in the bottom right term, is given
by the \textit{toric Jacobian} \cite{cox1994toricresidues}. 
The determinant of the Weyman complex yields the 
resultant, where higher multiplicities of the resultant may also occur. 
Our approach generalizes this result insofar as we
present an algorithm to determine the Weyman complex of systems with mixed support sets and
any twisted complexes with various higher sheaf cohomologies.
 
D'Andrea and Dickenstein \cite{DAndreaDickenstein01} consider systems of homogeneous polynomials.
They twist the associated Koszul complex on a projective space with any degree, thus incorporating
-- from our geometric point of view --
higher sheaf cohomology groups and non-trivial blocks of higher order in the differentials of 
the Weyman complex. 
However, these blocks are constructed from the homogeneous components of a Bézoutian, 
which is defined as in the classical Bézout formula for the resultant of two polynomials. 
%This is facilitated by relating the toric Jacobian and the B\'ezoutian for the specific ``critical degree''. 
This allows a Weyman complex to be described without direct reference to higher sheaf cohomologies. 

One may attempt to generalize these constructions. For instance, using a specific ``critical degree'' 
for the twist, the results \cite{Cattani1998} apply, which allows D'Andrea and Dickenstein to 
relate the toric Jacobian and the Bézoutian in the case of homogeneous polynomials. 
Cox and Dickenstein \cite{cox2004codimension}
consider polynomial systems which have mixed supports and apply various vanishing theorems 
for specific twists 
to arrive at a particularly simple first page of the the spectral sequence (\ref{eqn:TheSpectralSequence}),
which contains only one single non-trivial higher cohomology group. 
For a straightforward generalization of the toric Jacobian, the rank of this higher cohomology group would
have to be one, but this is not always the case. Therefore, we suspect that the use of 
Bézoutians in the construction of Weyman complexes cannot be generalized in an obvious way.
\section{An algorithmic construction of the Weyman complex}

The methods described in \cite{GelfandKapranovZelevinsky94} to compute $A$-resultants 
rely on coherent sheaves of modules on toric varieties. When the toric variety has 
no torus factors, these can be modeled via 
finitely generated graded modules over the Cox ring, cf. 
\cite[Propositions 5.3.9 and 6.A.4]{CoxLittleSchenk10}. 
We will first review the common methods 
to compute sheaf cohomology in this setup and then proceed to our construction 
of the Weyman functor and its implementation.

\subsection{Toric Varieties and {\v C}ech Cohomology}

%We let again $X$ be a toric variety with Cox ring $S$ and 
%$S_R = S\otimes_{\mathbf k} R$ for some finitely generated $\mathbf k$-algebra 
%$R$. 
Let $X$ be a toric variety with Cox ring $S$. 
For a finitely generated $S$-module $M$ we write 
\[
  \begin{xy}
    \xymatrix{
      0 \ar[r] & 
      \check C^0(M) \ar[r] &
      \check C^1(M) \ar[r] &
      \check C^2(M) \ar[r] &
      \dots
    }
  \end{xy}
\]
for the \v Cech complex associated to $M$ and its twists and the standard monomial 
generators $f_1,\dots, f_r \in S$ of the irrelevant ideal 
of $X$; cf. \cite[Section 9.5]{CoxLittleSchenk10}.
Every term in this complex is a graded $S$-module 
\[
  \check C^p(M) = \bigoplus_{0 < i_0 < \dots < i_p \leq r} M \otimes_{S} 
  S[f_{i_0}^{-1}, \dots, f_{i_p}^{-1}],
\]
but due to the localizations, these modules can not be assumed to be 
finitely generated. However, this complex can be written as a direct limit 
of complexes of finite $S$ modules as
\[
  \check C^\bullet(M) = \lim_{e \in \NN^r} \check C_{\leq e}^\bullet(M), 
  \qquad
  \check C^p_{\leq e}(M) = \bigoplus_{0 < i_0 < \dots < i_p \leq r} M \otimes_{S} 
  \frac{1}{f_{i_0}^{e_{i_1}} \cdots f_{i_p}^{e_{i_p}}}.
\]
The geometric \v Cech complex for the sheaf $\widetilde M$ on $X$ 
associated to $M$ is then given by the degree-$0$-part 
\[
  \check C^\bullet (\widetilde M) = \check C^\bullet(M)_0 = \lim_{e \in \NN^r} 
  \check C^\bullet_{\leq e}(M)_0,
\]
see \cite[Theorem 9.5.2]{CoxLittleSchenk10} and \cite[Theorem 9.5.7]{CoxLittleSchenk10}. 
In practice it is favorable to compute sheaf cohomology not through the \v Cech complex, 
but via graded strands of $\Ext$-modules as explained in 
\cite[Theorems 9.5.4 - 5]{CoxLittleSchenk10}.
For the latter we have a similar structure via direct limits 
induced by the natural inclusion $I^{[e]} \subset I^{[d]}$ for $e \geq d$ 
over the same index set $\NN^r$; see e.g. \cite{Hartshorne67}. 
In particular the existence of a morphism
$\mathrm{Kosz}(f_1^{e_1}, \dots, f_r^{e_r}) \to P^\bullet_{[e]}$
gives rise to a system of morphisms 
$\Hom(P^\bullet_{[e]}, M^\bullet) \to \check C^\bullet_{\leq e}(M^\bullet)$ 
for any free resolution $P^\bullet_{[e]}$ of the Frobenius power $I^{[e]}$ 
and this induces a quasi-isomorphism 
\[
  \lim_{e \in \NN^r} \Hom(P^\bullet_{[e]}, M)_0 \cong_{\textnormal{qis}} \check C^\bullet(M)_0
\]
with the \v Cech complex. 
See \cite[Lemma 9.5.8]{CoxLittleSchenk10} for the existence of truncations of these limits 
which capture the cohomology.
We will henceforth often work with the \v Cech complex, 
mostly for the convenience of notation and geometric accessibility. 
However, the reader should keep in mind that we may usually replace $C_{\leq e}^q(M^p)_0$ by 
$\Hom(P^q_{[e]}, M^p)_0$ in most arguments and we use the latter in 
most concrete computations. 

\medskip
For a product $X \times_{\mathbf k} \Spec R$ the above considerations carry 
over almost verbally to graded \textit{free} modules over the ring 
$S_R = S \otimes_{\mathbf k} R$ 
and the projection $\pi \colon X \times_{\mathbf k} \Spec R \to \Spec R$ 
via base change; 
cf. \cite[Theorem 2.3.36]{Dimca04}. The general case can then be reduced 
to complexes of free modules by choice of a Cartan-Eilenberg resolution. 

%We use strands of local cohomology modules to compute cohomology of 
%coherent sheaves on toric varieties; see e. g. \cite[Section 9.5]{CoxLittleSchenk10}. 
%It is explained in Theorem \cite[Theorem 9.5.2]{CoxLittleSchenk10} how the
%``local {\v C}ech complex'' relates to local cohomology modules and in 
%\cite[Theorems 9.5.4 - 5]{CoxLittleSchenk10} how $\Ext$-modules can be used to 
%compute these. The actual sheaf cohomology groups can then be extracted as graded 
%strands for the local cohomology at the ``irrelevant ideal'' 
%according to \cite[Theorem 9.5.7]{CoxLittleSchenk10}. Finally 
%\cite[Lemma 9.5.8]{CoxLittleSchenk10} assures the stabilization of the 
%strands of $\Ext$-modules. 

\medskip
The following theorem summarizes the results on direct images we shall need 
for our purposes.

\begin{theorem}
  \label{thm:DirectImageForToricVarities}
  Let $X$ be a toric variety without torus factors over a field $\mathbf k$ 
  and with Cox ring $S$. 
  Let $R$ be a commutative Noetherian ring and consider the direct product 
  \[
    \begin{xy}
      \xymatrix{
        & \tilde X = X \times_{\mathbf k}\Spec R \ar[dl]_\lambda \ar[dr]^\pi & \\
        X & & \Spec R
      }
    \end{xy}
  \]
  We set $S_R := S \otimes_{\mathbf k} R$ and $\tilde I = I\cdot S_R$ 
  for the irrelevant ideal $I = \langle f_1, \dots, f_r\rangle \subset S$ of $X$.
  Then for 
  any right bounded complex 
  \[
    M^\bullet : 
    \begin{xy}
      \xymatrix{
        \cdots \ar[r] &
        M^{p_0-2} \ar[r] &
        M^{p_0-1} \ar[r] &
        M^{p_0} \ar[r] &
        0 & 
      }
    \end{xy}
  \]
  of free graded $S_R$-modules 
  %$M_i = \bigoplus_{j=0}^{r_i} S_R[-\alpha_{j, k}]$ 
  and any homological degree $q$ 
  there exists $e \in \NN^r$ such that the complex
  \begin{equation}
    \label{eqn:TotalComplexStrand}
    \left(\mathrm{Tot}\left(\Hom(P_{[e]}^\bullet, M^\bullet)\right)\right)_0
  \end{equation}
  is quasi-isomorphic to the direct image $R\pi_*\left(\tilde M^\bullet\right)$ 
  of the complex of sheaves of $\mathcal O_{\tilde X}$-modules associated to 
  $M^\bullet$ from 
  degree $q$ onwards, where $P_{[e]}^\bullet$ is a graded resolution of the ideal 
  $\tilde I^{[e]} = \langle f_1^{e_1}, \dots, f_r^{e_r}\rangle$.
  %\footnote{For an ideal $I = \langle f_1, \dots, f_r\rangle$ 
  %we denote by $I^{[e]} = \langle f_1^e, \dots, f_r^e\rangle$ its Frobenius power.}
\end{theorem}

By ``quasi-isomorphic from degree $q$ onwards'' we mean that there exists a roof 
of morphisms of complexes whose induced maps in cohomology are isomorphisms 
in all cohomological degrees $\geq q$. 
Note that 
if the complex $M^\bullet$ is bounded, then we need not specify the homological 
degree $q$ and choose a global $e \in \NN^r$. 
%Moreover, by choice of a Cartan-Eilenberg resolution, this theorem 
%can be applied to an arbitrary right bounded complex $M^\bullet$ of finitely 
%generated graded $S_R$-modules. 

\begin{proof}
  We may assume $M^i = \bigoplus_{k=0}^{r_i} S_R[-\alpha_{i, k}]$ 
  for some degrees $\alpha_{i, k} \in \Cl(X)$. Explicit
  choices of $e$ for a single free module $S[-\alpha_{i,k}]$ are given in 
  \cite{EisenbudMustataStillman2000}; see also
  \cite[Theorem~9.5.10]{CoxLittleSchenk10}. 
  Note that taking direct images commutes with flat
  base change, so the results immediately carry over to 
  $S_R[-\alpha_{i, k}]$. Having this established, we may choose $e$ to be the
  component-wise supremum of
  all those bounds involved in the free $S_R$-modules which are relevant
  for the cohomology of the direct image up to degree $q$.
\end{proof}

While the complex (\ref{eqn:TotalComplexStrand}) is straightforward to write down, 
its size is usually too big to be of any practical use in computations. After all, 
every strand of degree $0$ has a monomial basis whose length tends to explode with 
an increasing number of variables of $S_R$ and degrees $-\alpha_{i, k}$ appearing in the 
$M^i$. In particular, when choosing a global $e$ as proposed in Theorem 
\ref{thm:DirectImageForToricVarities}, high degrees $\alpha_{i, k}$ for $i \ll p_0$ 
push the size of the monomial bases for terms coming from $M^{p_0}$ up very quickly. 

\medskip

\subsection{Gau{\ss} reduction for complexes}
\label{sec:SimplificationUpToHomotopy}
  Consider a complex $K^\bullet$ of free $R$-modules over a commutative ring $R$
  and assume that the matrix for the boundary map $\varphi$ in some homological 
  degree $p$ contains an invertible 
  submatrix $A \in R^{r \times r}$ for some rank $r \leq m, n$. Then we 
  can use Gau{\ss} reduction for the rows and columns to obtain a
  change of basis in domain and codomain so that $\varphi$ takes 
  a block diagonal form with respect to these bases. We obtain 
  the complex in the second row of the diagram below. 
  \[
    \begin{xy}
      \xymatrix{
        & & p & p+1 & \\
        K^\bullet : &
        \cdots \ar[r] & 
        R^m \ar[r] \ar@/^0.5pc/[d]^{+} 
        \ar[r]^{
          \left(
          \begin{array}{c|c}
            A & B \\
            \hline
            C & D
          \end{array}
          \right)
        } 
        &
        R^n \ar@/^0.5pc/[d]^{
          \left(
          \begin{array}{c|c}
            1 & \mp A^{-1}B \\
            \hline
            0 & 1
          \end{array}
          \right)
        }_{-}
        \ar[r]
        &
        \cdots \\
        &
        \cdots \ar[r] \ar[dr]&
        R^r \oplus R^{m-r} 
        \ar[r]_{
          \left(
          \begin{array}{c|c}
            A & 0 \\
            \hline
            0 & \tilde D % - CA^{-1}B
          \end{array}
          \right)
        } 
        \ar@/^0.5pc/[u]_{-}^{
          \left(
          \begin{array}{c|c}
            1 & 0 \\
            \hline
            \mp CA^{-1} & 1
          \end{array}
          \right)
        }
        \ar@/^0.5pc/[d]_\rho &
        R^r \oplus R^{n-r} \ar@/^0.5pc/[u]^{+} \ar@/^0.5pc/[d]^\rho 
        \ar[r]
        &
        \cdots \\
        \tilde K^\bullet : &
        &
        R^{m-r} \ar@/^0.5pc/[u]^\iota 
        \ar[r]_{\tilde D} &
        R^{n-r} \ar@/^0.5pc/[u]_\iota \ar[ur]
        & 
      }
    \end{xy}
  \]
  If we are only interested in the complex up to homotopy (or quasi-isomorphism), 
  we may split off the homologically trivial subcomplex
  $R^r \overset{A}{\longrightarrow} R^r$ to arrive at the complex $\tilde K^\bullet$ 
  given by the deviation through the bottom row with the matrix $\tilde D = D - CA^{-1}B$. 
  Denote the projection to and the inclusion of $\tilde K^\bullet$ by 
  \[
    \begin{xy}
      \xymatrix{
        K^\bullet \ar[r]^\rho & 
        \tilde K^\bullet, \ar@/^1pc/[l]^\iota
      }
    \end{xy}
  \]
  respectively. Then the block diagonal matrix 
  \[
    \left(
    \begin{array}{c|c}
      A^{-1} & 0 \\
      \hline
      0 & 0
    \end{array}
    \right)
    \colon R^n \to R^m
  \]
  yields a homotopy $h \colon K^\bullet \to K^{\bullet}[-1]$ so that 
  \[
    h \circ \varphi + \varphi \circ h = \mathrm{id} - \iota \circ \rho.
  \]
  Here we extend $h$ by the zero map in all homological degrees which are 
  not involved in the above reduction.

  Having explained the reduction for one single differential in a complex, 
  we may proceed with the complex $\tilde K^\bullet$ in place of $K^\bullet$ 
  and repeat the process until none of the differentials has a representing matrix 
  containing units. Note that, unless $R$ is local or suitably graded over a field, 
  this process can not be assumed to be deterministic; it involves several arbitrary 
  choices, including the order of the maps to be considered. 

\begin{example}
  \label{exp:CechComplexReduction}
  The main application for Gau{\ss} reduction of complexes in this article will be 
  on \v Cech complexes on toric varieties. Let $S$ be the Cox ring of a toric 
  variety $X$ with no torus factors over $\mathbf k = \QQ$. 
  Consider the module $M = S[-\alpha]$ 
  for some $\alpha \in \mathrm{Cl}(X)$ and let 
  $\check C^\bullet_{\leq e}(S[-\alpha])_0$ be its truncated \v Cech complex. 
  For $e_0 \in \NN^r$ sufficiently big we have 
  \[
    H^q(X, \mathcal O(-\alpha)) \cong 
    H^q\left(\check C^\bullet_{\leq e}(S[-\alpha]) \right)_0
    =: \check H^q_{\leq e}(S[-\alpha])_0
    \cong
    \check H^q_{\leq e}(S)_{\alpha}
  \]
  for every $e \geq e_0$. 
  Since we are working over a field $\QQ$, Gau{\ss} reduction of complexes 
  provides us with quasi-isomorphisms $\iota$ and $\rho$ in a commutative 
  diagram 
  \begin{equation}
    \label{eqn:SimplificationUpToHomotopyRank1Module}
    \begin{xy}
      \xymatrix{
        \cdots \ar[r]_-{\check \partial} &
        \check C_{\leq e}^{i-1}(S)_\alpha 
        %\Hom(P^{[e]}_{i+1}, S_R)_{\beta} 
        \ar[r]_-{\check \partial} \ar[d]^\rho \ar@/_1pc/[l]_h & 
        \check C_{\leq e}^{i}(S)_\alpha 
        %\Hom(P^{[e]}_{i+1}, S_R)_{\beta} 
        \ar[r]_-{\check \partial} \ar[d]^\rho \ar@/_1pc/[l]_h & 
        \check C_{\leq e}^{i+1}(S)_\alpha 
        %\Hom(P^{[e]}_{i+1}, S_R)_{\beta} 
        \ar[r]_-{\check \partial} \ar[d]^\rho \ar@/_1pc/[l]_h & 
        \cdots \ar@/_1pc/[l]_-h \\
        \cdots \ar[r]^-0 &
        \check H_{\leq e}^{i-1}(S)_\alpha 
        %\Hom(P^{[e]}_{i+1}, S_R)_{\beta} 
        \ar[r]^-{0} \ar[u]^\iota & 
        \check H_{\leq e}^{i}(S)_\alpha 
        %\Hom(P^{[e]}_{i+1}, S_R)_{\beta} 
        \ar[r]^-{0} \ar[u]^\iota & 
        \check H_{\leq e}^{i+1}(S)_\alpha 
        %\Hom(P^{[e]}_{i+1}, S_R)_{\beta} 
        \ar[r]^-{0} \ar[u]^\iota & 
        \cdots  \\
        \cdots &
        \check H_{\leq e_0}^{i-1}(S)_\alpha 
        %\Hom(P^{[e]}_{i+1}, S_R)_{\beta} 
        \ar[u]^\cong & 
        \check H_{\leq e_0}^{i}(S)_\alpha 
        %\Hom(P^{[e]}_{i+1}, S_R)_{\beta} 
        \ar[u]^\cong & 
        \check H_{\leq e_0}^{i+1}(S)_\alpha 
        %\Hom(P^{[e]}_{i+1}, S_R)_{\beta} 
        \ar[u]^\cong & 
        \cdots  \\
      }
    \end{xy}
  \end{equation}
together with a corresponding homotopy $h$. Note that while 
the form, which these maps take, depends on the choice of pivots 
in the Gau{\ss} reduction, the isomorphisms with any set of chosen 
reference vector spaces $\check H^q_{\leq e_0}(S)_\alpha$ 
for the cohomology groups are \textit{natural}.
For different exponent vectors $e \leq e' \in \NN^r$, however, 
we can not assume any compatibility of the respective 
homotopies $h$ with the natural maps 
$\check C^q_{\leq e}(S)_\alpha \to \check C^q_{\leq e'}(S)_\alpha$, 
i.e. that they were to form commutative diagrams.

The same considerations apply to the complexes $\Hom(P^\bullet_{[e]}, S)_\alpha$ 
and the computation of sheaf cohomology via graded strands of $\Ext$-modules.
\end{example}

\subsection{Simplification of total complexes} 

Consider the total complex of the truncated {\v C}ech double 
complex $\check C^\bullet_{\leq e}(M^\bullet)_0$ 
for a complex of free graded $S_R$-modules $(M^\bullet, \varphi^\bullet)$ 
of finite, free, graded $S_R$-modules $M^p = \bigoplus_{j=1}^{r_p} S_R[-\alpha_{p, j}]$
\begin{equation}
  \label{eqn:CechDoubleComplex}
  \begin{xy}
    \xymatrix{
      & \vdots & \vdots & \vdots & \vdots & \\
      \cdots \ar[r]^-\varphi &
      \check C^2_{\leq e}(M^{p-2})_0 \ar[u]^{\check \partial} \ar[r]^-\varphi & 
      \check C^2_{\leq e}(M^{p-1})_0 \ar[u]^{\check \partial} \ar[r]^-\varphi & 
      \check C^2_{\leq e}(M^{p  })_0 \ar[u]^{\check \partial} \ar[r]^-\varphi & 
      \check C^2_{\leq e}(M^{p+1})_0 \ar[u]^{\check \partial} \ar[r]^-\varphi & 
      \cdots \\
      \cdots \ar[r]^-\varphi &
      \check C^1_{\leq e}(M^{p-2})_0 \ar[u]^{\check \partial} \ar[r]^-\varphi & 
      \check C^1_{\leq e}(M^{p-1})_0 \ar[u]^{\check \partial} \ar[r]^-\varphi & 
      \check C^1_{\leq e}(M^{p  })_0 \ar[u]^{\check \partial} \ar[r]^-\varphi & 
      \check C^1_{\leq e}(M^{p+1})_0 \ar[u]^{\check \partial} \ar[r]^-\varphi & 
      \cdots \\
      \cdots \ar[r]^-\varphi &
      \check C^0_{\leq e}(M^{p-2})_0 \ar[u]^{\check \partial} \ar[r]^-\varphi & 
      \check C^0_{\leq e}(M^{p-1})_0 \ar[u]^{\check \partial} \ar[r]^-\varphi & 
      \check C^0_{\leq e}(M^{p  })_0 \ar[u]^{\check \partial} \ar[r]^-\varphi & 
      \check C^0_{\leq e}(M^{p+1})_0 \ar[u]^{\check \partial} \ar[r]^-\varphi & 
      \cdots \\
    }
  \end{xy}
\end{equation}
The differential of the total complex in homological degree $p$ has a 
block structure for the decomposition
\[
  \check C^{0}_{\leq e}(M^{p})_0 \oplus
  \check C^{1}_{\leq e}(M^{p+1})_0 \oplus
  \check C^{2}_{\leq e}(M^{p+2})_0 \oplus
  \cdots
  \longrightarrow
  \check C^{0}_{\leq e}(M^{p-1})_0 \oplus
  \check C^{1}_{\leq e}(M^{p})_0 \oplus
  \check C^{2}_{\leq e}(M^{p+1})_0 \oplus
  \cdots
\]
of its domain and codomain 
and the representing matrix of this differential with respect to any 
chosen basis for the blocks will take the form 
\[
  \Phi^p = 
  \left(
  \begin{array}{c|c|c|c|c}
    (-1)^{p-1}\varphi^{p, 0} & \check \partial^{p, 0} & & & \\
    \hline
    & (-1)^{p-1}\varphi^{p-1, 1} & \check \partial^{p-1, 1} & & \\
    \hline
    & & (-1)^{p-1}\varphi^{p-2, 2} & \check \partial^{p-2, 2}& \\
    \hline
    & & &\ddots & \ddots\\
  \end{array}
  \right).
\]
%\[
%  \Phi^p = 
%  \left(
%  \begin{array}{c|c|c|c}
%    (-1)^{p-1}\varphi^{p, 0} & & & \\
%    \hline
%    \check \partial^{p, 0} 
%    & (-1)^{p-1}\varphi^{p-1, 1} & & \\
%    \hline
%    &\check \partial^{p-1, 1} & (-1)^{p-1}\varphi^{p-2, 2} & \\
%    \hline
%    & & \check \partial^{p-2, 2}& \ddots \\
%    \hline
%    & & & \ddots \\
%  \end{array}
%  \right).
%\]
We will refer to this as the \textit{macro block structure} of the differential 
of the total complex.

\medskip
Each one of those macro blocks has again a block structure. Namely, since 
by assumption $M^i = \bigoplus_{k=0}^{r_i} S_R[-\alpha_{i, k}]$ for some degrees 
$\alpha_{i, k} \in \mathrm{Cl}(X)$ we have 
$\check C^j_{\leq e}(M^i)_0 \cong \bigoplus_{k=1}^{r_i} \check C^j_{\leq e}(S_R)_{\alpha_{i, k}}$.
The representing matrix for $\check \partial^{i, j}$ will then be of block diagonal 
form 
\[
  \left(
  \begin{array}{c|c|c|c}
    \check \partial^j_{\leq e}(\alpha_{i, 1}) & & & \\
    \hline
    & \check \partial^{j}_{\leq e}(\alpha_{i, 2}) & & \\
    \hline
    & & \check \partial^{j}_{\leq e}(\alpha_{i, 3}) & \\
    \hline
    & & & \ddots
  \end{array}
  \right)
\]
where $\check \partial^j_{\leq e}(\alpha)$ is the $j$-th differential of the 
$\alpha$-strand of the truncated \v Cech complex $\check C^\bullet_{\leq e}(S_R)_{\alpha}$. 
This will be referred to as the \textit{micro block structure} 
of the differential of the total complex. 

We wish to simplify the total complex of $\check C^\bullet_{\leq e}(M^\bullet)_0$ 
up to homotopy as recalled in 
Section \ref{sec:SimplificationUpToHomotopy}. In this process we are free 
to choose the invertible submatrices for the Gau{\ss} reduction to our liking. 
As was explained using Diagram (\ref{eqn:SimplificationUpToHomotopyRank1Module}), 
the representing matrix in a monomial basis of any micro block 
$\check \partial^j(\alpha_{i, k})$ for $\check \partial^{i, j}$ 
is of the form 
\[
  \left(
  \begin{array}{c|c}
    A & B \\
    \hline
    C & D
  \end{array}
  \right)
  \sim
  \left(
  \begin{array}{c|c}
    A & 0 \\
    \hline
    0 & 0
  \end{array}
  \right)
\]
for some invertible matrix $A$ and this can be brought to the form on the right 
hand side using row and column operations. This third level of block nesting 
will be referred to as the \textit{nano block structure}. If we consider the right hand side 
as a micro block in the matrix for the differential of the total complex, then 
the row and column operations used for reduction of the micro block extend to the ambient matrix 
\begin{eqnarray*}
  \left(
  \begin{array}{c|c|c|c|c|c}
    \ddots & \ddots & & & &\\
    \hline
    &(-1)^{p-1}\varphi^{p-i+1, i-1}& \check \partial^{p-i+1, i-1} & & & \\
    \hline
    & &(-1)^{p-1} \varphi^{p-i, i} & 
    \begin{array}{c|c|c}
      \begin{array}{c|c}
        A & B \\
        \hline
        C & D
      \end{array}
      & & \\
      \hline 
      & \check \partial^{i}_{\leq e}(\alpha_{p-i, 2}) & \\
      \hline
      & & \ddots \\
    \end{array}
    & \\
    \hline 
    & & &(-1)^{p-1} \varphi^{p-i-1, i+1} & \check \partial^{p-i-1, i+1} & \\
    \hline
    & & & &\ddots & \ddots\\
  \end{array}
  \right) 
  & & \\
  \qquad
  \sim \qquad
  \left(
  \begin{array}{c|c|c|c|c|c}
    \ddots & \ddots & & & & \\
    \hline
    &(-1)^{p-1} \varphi^{p-i+1, i-1} & \check \partial^{p-i+1, i-1} & & & \\
    \hline
    & & 
    \phi
    &
    \begin{array}{c|c|c}
      \begin{array}{c|c}
        A & 0 \\
        \hline
        0 & 0
      \end{array}
      & & \\
      \hline 
      & \check \partial^{i}_{\leq e}(\alpha_{p-i, 2}) & \\
      \hline
      & & \ddots \\
    \end{array}
    & \\
    \hline 
    & &  
    \kappa
    & 
    \psi
    & \check \partial^{p-i-1, i+1} & \\
    \hline
    & & & &\ddots & \ddots\\
  \end{array}
  \right)
\end{eqnarray*}
where now the matrix for $\phi$ has zeroes in the rows of the nano block 
for $A$ and $\psi$ has zeroes in the columns of the nano block for $A$. 
Moreover, we see a new block $\kappa$ emerging from those operations. 

We may apply this Gau{\ss} reduction for all the micro blocks for the 
$\check \partial_{\leq e}^{i-1}(\alpha_{p-i+1, k})$, $i=1, 2, \dots$
appearing in the matrix of the 
differential $\Phi_p$. As a result, we obtain matrices for the base 
change in domain and codomain to arrive at our new matrix $\Phi_p'$ 
after the reduction:
\begin{eqnarray*}
  & & 
  \left(
  \begin{array}{c|c|c|c|c|c}
    S^{-1} & 0 & 0 & \cdots & 0 & 0 \\
    \hline
    -S^{-1} \cdot \varphi \cdot h & S^{-1} & 0 & \cdots & 0 & 0 \\
    \hline
    S^{-1} \cdot \varphi \cdot h \cdot \varphi \cdot h & 
    S^{-1} \cdot \varphi \cdot h & 
    S^{-1} & \ddots & \vdots & \vdots \\
    \hline 
    \vdots & \ddots & \ddots & \ddots & 0 & 0 \\
    \hline 
    \pm S^{-1} \cdot \varphi \cdot h \cdots \varphi \cdot h & 
    \cdots &
    S^{-1} \cdot \varphi \cdot h \cdot \varphi \cdot h & 
    - S^{-1} \cdot \varphi \cdot h & 
    S^{-1} & 0
    \\
    \hline
    \mp \varphi \cdot h \cdots \varphi \cdot h & \dots & 
    \varphi \cdot h \cdot \varphi \cdot h \cdot \varphi \cdot h &
    \varphi \cdot h \cdot \varphi \cdot h &-\varphi \cdot h & \mathbf 1
  \end{array}
  \right)
  \cdot
  \left(
  \begin{array}{c|c|c|c|c}
    \varphi & \check \partial & & & \\
    \hline
    & \varphi & \ddots & & \\
    \hline
    & & \ddots & \check \partial & \\
    \hline
    & & & \varphi & \check\partial\\
    \hline
    & & & & \varphi
  \end{array}
  \right)
  \\
  \iffalse
  \left(
  \begin{array}{c|c|c|c|c}
    \varphi' & \Delta & & & \\
    \hline
    \kappa& \varphi' & \Delta & & \\
    \hline
    \vdots & \ddots & \ddots & \ddots & \\
    \hline
    \vdots & & \ddots & \varphi' & \Delta\\
    \hline
    \kappa& \cdots & \cdots & \kappa & \varphi' 
  \end{array}
  \right)
  \fi
  & & = 
  \left(
  \begin{array}{c|c|c|c}
    \psi & \Delta & & \\
    \hline
    \psi & \psi & \ddots & \\
    \hline
    \vdots & \ddots & \ddots & \Delta \\
    \hline
    \psi & \cdots & \psi & \psi
  \end{array}
  \right)
  \cdot
  \left(
  \begin{array}{c|c|c|c|c|c}
    \mathbf 1 & 0 & 0 & \cdots & 0 & 0\\
    \hline
    -h \cdot \varphi& T & 0 & \cdots & 0 & 0\\
    \hline
    h \cdot \varphi \cdot h \cdot \varphi& 
    -h \cdot \varphi \cdot T & T & \ddots & \vdots & \vdots \\
    \hline
    - h \cdot \varphi \cdot h \cdot \varphi \cdot h \cdot \varphi& 
    h \cdot \varphi \cdot h \cdot \varphi \cdot T& 
    -h \cdot \varphi \cdot T & \ddots  & 0 & 0 \\
    \hline
    \vdots & \vdots & \ddots & \ddots & T & 0 \\
    \hline
    \pm h \cdot \varphi \cdots h \cdot \varphi &
    \mp h \cdot \varphi \cdots h \cdot \varphi \cdot T& \cdots & 
    h \cdot \varphi \cdot h \cdot \varphi \cdot T & - h \cdot \varphi \cdot T & T \\
  \end{array}
  \right)^{-1}
\end{eqnarray*}
Here we omitted the numerous indices of the matrices appearing in the blocks. Which matrix is actually meant 
should be clear from its positional context. We have 
\begin{itemize}
  \item $\varphi$ the macro blocks for the induced horizontal maps in the \v Cech double complex.
  \item $\check \partial$ the block diagonal matrices with micro blocks 
    $\left(\begin{smallmatrix} A & B \\ C & D\end{smallmatrix}\right)$ 
    from the \v Cech complexes
  \item $S^{-1}$ the block diagonal matrix whose micro blocks read 
    $\left(\begin{smallmatrix} 1 & 0 \\ -CA^{-1} & 1\end{smallmatrix}\right)$ 
    for the Gau{\ss} reduction of the respective micro block of $\check \partial$.
  \item $h$ the block diagonal matrix whose micro blocks are 
    $\left(\begin{smallmatrix} A^{-1} & 0 \\ 0 & 0 \end{smallmatrix}\right)$ 
    for the same reduction as explained in Section \ref{sec:SimplificationUpToHomotopy}.
  \item $T$ the block diagonal matrix whose micro blocks read 
    $\left(\begin{smallmatrix} 1 & -A^{-1}B \\ 0 & 1\end{smallmatrix}\right)$ 
    for the respective column reductions.
  \item $\Delta$ the block diagonal matrices from the \v Cech complexes 
    after reduction with micro blocks 
    $\left(\begin{smallmatrix} A & 0 \\ 0 & 0 \end{smallmatrix}\right)$.
  \item $\psi$ the induced macro blocks after Gau{\ss} reduction. By construction, the rows and columns 
    of nano blocks $A$ appearing in the $\Delta$ are zero in $\psi$.
\end{itemize}
Note that only due to format restrictions we have written the matrix for column reduction with an inverse sign 
on the right hand side of the equation. 

The inverses of the respective base change matrices above are
\begin{eqnarray*}
  \left(
  \begin{array}{c|c|c|c|c}
    S & & & & \\
    \hline
    \varphi \cdot h & S & & & \\
    \hline
    0 & \varphi \cdot h & \ddots & & \\
    \hline
    \vdots & \ddots & \ddots & S & \\
    \hline
    0 & \cdots & 0 & \varphi \cdot h & \mathbf 1
  \end{array}
  \right)
  &\textnormal{ and }&
  \left(
  \begin{array}{c|c|c|c|c}
    \mathbf 1 & & & & \\
    \hline
    h \cdot \varphi & T^{-1} & & & \\
    \hline
    0 & h \cdot \varphi & T^{-1} & & \\
    \hline
    \vdots & \ddots & \ddots & \ddots & \\
    \hline
    0 & \cdots & 0 & h \cdot \varphi & T^{-1}
  \end{array}
  \right).
\end{eqnarray*}

We may perform the same 
reductions on the neighboring matrices in the total complex to obtain 
a combined base change 
\[
  U^{-1} = T_{\mathrm{in}}^{-1} \cdot S_{\mathrm{out}}^{-1} 
  \qquad \textnormal{ and } 
  \qquad
  U = S_{\mathrm{out}} \cdot T_{\mathrm{in}}
\]
where we denote by $T_{\mathrm{out}}$ the base change in the codomain 
of the incoming and by $S_{\mathrm{out}}$ the base change in the domain 
of the outgoing map. 
Unfortunately, these matrices do not admit a reasonable 
closed form to write down here. 

\medskip
Denote the homotopies of the induced Gau{\ss} reduction of the total complex of 
(\ref{eqn:CechDoubleComplex}) by $h$ and the inclusion of and the projection 
to the simplified complex by $\iota$ and $\rho$. The reader may then verify the 
following Lemma along the above lines.
\begin{lemma}
  \label{lem:TotalComplexRaw}
  The Gau{\ss} reduction of the total complex of (\ref{eqn:CechDoubleComplex})
  induced by the Gau{\ss} reduction of the strands of the \v Cech complexes  
  $\check C^\bullet_{\leq e}(S_R)_{-\alpha_{i, k}}$ in the columns of the double 
  complex results in a complex $W^\bullet$ with terms 
  \[
    W^i = 
    %\bigoplus_{p+q = i} \bigoplus_{k=1}^{r_p} 
    %\check H^q_{\leq e}(X, \OO(-\alpha_{p, k})) \otimes_{\mathbf k} R 
    \bigoplus_{p+q = i} \bigoplus_{k=1}^{r_p} 
    \check H^q_{\leq e}(S_R[-\alpha_{p, k}])_0
    =
    \bigoplus_{p+q = i} \bigoplus_{k=1}^{r_p} 
    \check H^q_{\leq e}(S)_{\alpha_{p, k}} \otimes_{\mathbf k} R 
  \]
  and morphisms $\phi^i \colon W^i \to W^{i+1}$ whose blocks with respect to the 
  decomposition into direct sums are given by 
  \[
    \phi^{i, r}_{p, q} \colon 
    \check H^q_{\leq e}(S_R[-\alpha_{p, k}])
    \to
    \check H^{q-r+1}_{\leq e}(S_R[-\alpha_{p+r, k}]),
    \qquad
    \phi^{i, r}_{p, q} = (-1)^{(i-1)(r-1)} \rho \circ \underbrace{\varphi \circ h \circ \cdots 
    \circ \varphi \circ h}_{r-1 \textnormal{ times }} \circ \varphi \circ \iota.
  \]
\end{lemma}
To assist the reader in their understanding, let us remark that 
the individual blocks $\phi_{p, q}^{i, r}$ arise for $r = 1, 2, \dots$ 
by tracing an element $v \in \check H^q_{\leq e}(S_R[-\alpha_{p, k}])$ 
while we ``walk down the staircase'' through the following diagram and 
its analog for the ring $S_R$. 
\begin{equation}
  \label{eqn:StaircaseLifting}
  \begin{xy}
    \xymatrix{
      % zeroth row
      &
      \bigoplus_{k = 1}^{r_{p+1}} \check H^{q}_{\leq e}(S[-\alpha_{p+1, k}])_0
      & \\
      % first row
      \bigoplus_{k = 1}^{r_p} \check C^q_{\leq e}(S[-\alpha_{p, k}])_0 
      \ar[r]^{\pm\varphi} &
      \bigoplus_{k = 1}^{r_{p+1}} \check C^q_{\leq e}(S[-\alpha_{p+1, k}])_0 
      \ar@/_1pc/[d]_{h}
      \ar[u]^\rho 
      &
      \bigoplus_{k = 1}^{r_{p+2}} \check H^{q-1}_{\leq e}(S[-\alpha_{p+2, k}])_0
      \\
      % second row
      \bigoplus_{k = 1}^{r_p} \check H^q_{\leq e}(S[-\alpha_{p, k}])_0
      \ar@{^{(}->}[u]^\iota 
      & 
      \bigoplus_{k = 1}^{r_{p+1}} \check C^{q-1}_{\leq e}(S[-\alpha_{p+1, k}])_0 
      \ar[u]_{\check \partial}
      \ar[r]^{\pm\varphi}
      &
      \bigoplus_{k = 1}^{r_{p+2}} \check C^{q-1}_{\leq e}(S[-\alpha_{p+2, k}])_0 
      \ar@/_1pc/[d]_{h}
      \ar[u]^\rho 
      \\
      % third row
      & & 
      \bigoplus_{k = 1}^{r_{p+2}} \check C^{q-2}_{\leq e}(S[-\alpha_{p+2, k}])_0 
      \ar[u]_{\check \partial}
      \ar[r]^-{{\pm\varphi}} &
      \cdots
      \\
    }
  \end{xy}
\end{equation}

\subsection{Detailed construction of the Weyman functor}
\label{sec:WeymanFunctorDetails}
The complex $W^\bullet$ from Lemma \ref{lem:TotalComplexRaw} is already close 
to the result of the Weyman functor applied to $M^\bullet$. To obtain 
$\mathbb W^\bullet (M^\bullet)$ we choose a function 
\begin{equation}
  \label{eqn:OptimalKFunction}
  e_{\min} \colon \mathrm{Cl}(X) \to \NN^r
\end{equation}
where $r$ is the number of monomial generators of the irrelevant ideal 
$I \subset S$ for $X$ such that 
\[
  H^q(X, \OO(-\alpha)) \cong \check H^q_{\leq e_{\min}(\alpha)}(S)_\alpha.
  %H^q(X, \OO(-\alpha)) \cong H^q\left(\Hom(P^\bullet_{[e_{\min}(\alpha)]}, S)_\alpha\right)
\]
%for a minimal resolution $P^\bullet_{[e_{\min}(\alpha)]}$ of 
%the Frobenius power $I^{[e_{\min}]}$. 
In particular, this will choose our 
\textit{reference model} for the sheaf cohomology of the twisting sheaves, i.e. 
we will set 
\begin{equation}
  \label{eqn:CohomologyModelChoice}
  R^q\pi_* \OO(-\alpha) := \check H^q_{\leq e_{\min}(\alpha)}(S)_{\alpha} \otimes_{\mathbf k} R
\end{equation}
with the cohomology on the right hand side computed via Gau{\ss} reduction 
of the $\alpha$-strand of the \v Cech complex 
$\check C^\bullet_{\leq e_{\min}(\alpha)}(S)_{\alpha}$. 

For our implementation we actually use the computation of direct 
images via graded strands of $\Ext$-modules and choose $e_{\min}$ as in 
\cite[Proposition 3.1]{EisenbudMustataStillman2000}
with the computation of support sets via $\Ext$ as in 
\cite[Corollary 1.2 (b)]{EisenbudMustataStillman2000}.
It seems possible to make reasonable refinements of this strategy 
by not only considering powers of the monomial generators of the 
irrelevant ideal, but also other co-finite limits of ideals. 
However, we have not yet explored this beyond products 
of projective spaces. 

\begin{remark}
  \label{rem:DirectImageChoice}
  Note that (\ref{eqn:CohomologyModelChoice}) 
  has little meaning in a strict mathematical sense, 
  since its left hand side is usually understood as the 
  cohomology of \textit{all} complexes representing the direct image, 
  but up to the natural induced identifications.
  However, for a practical computation and implementation of these structures, 
  the systematic choice of a unique representative of such equivalence classes 
  and realizing the explicit transition between different such representatives, 
  is crucial. 
\end{remark}

\subsubsection{Differentials in the Weyman complex}
\label{sec:DifferentialsInTheWeymanComplex}
The upshot of the considerations of the Gau{\ss} reduction of the total complex 
of (\ref{eqn:CechDoubleComplex}) in the previous section is that we may compute 
the terms and maps of $W^\bullet$ from Lemma \ref{lem:TotalComplexRaw} 
\textit{blockwise} and without even composing the matrices for the 
differentials of the total complex. 
In order to compute the differential 
\[
  \phi^i \colon \mathbb W^i(M^\bullet) \to \mathbb W^{i+1}(M^\bullet)
\]
we proceed as follows along the staircase in (\ref{eqn:StaircaseLifting}). Let 
$u \in R^q \pi_* \OO(-\alpha_{p, k_0}) \hookrightarrow \bigoplus_{k=1}^{r_p} 
\check C_{\leq e_0}^q (S[-\alpha_{p, k}])_0$
be a generator of $\mathbb W^i(M^\bullet)$ for 
$e_0 = e_{\min}(\alpha_{p, k_0})$ with $k_0$ the index of the summand of $u$. 
Then we set 
$v_1 = \varphi(u) \in \bigoplus_{k=1}^{r_{p+1}} \check C_{\leq e_0}^q (S[-\alpha_{p+1, k}])_0$
and 
$w_1 = \rho(v_1) \in \bigoplus_{k=1}^{r_{p+1}} \check H^q_{\leq e_0} (S[-\alpha_{p+1, k}])_0$.
Then we define inductively 
\[
  v_{r+1} = \varphi(h(v_r)), \quad w_{r+1} = \rho(v_{r+1}).
\]
Finally, we can make the identifications 
\[
  W^{i+1} \supset 
  \bigoplus_{k=1}^{r_{p+1}} \check H^{q-r+1}_{\leq e_0}(S_R[-\alpha_{p+r, k}])_0 
  \overset{\cong}{\longrightarrow}
  \bigoplus_{k=1}^{r_{p+1}} \check H^{q-r+1}_{\leq e_{\min}(-\alpha_{p+r, k})}(S_R[-\alpha_{p+r, k}])_0 
  \subset 
  \mathbb W^{i+1}(M^\bullet)
\]
explicit for all summands appearing in domain and codomain of $\phi^i$. 
Then $\phi_{p, q}^{i, r}(u)$ is the image of $w_r$ along these identifications. 

\begin{remark}
  \label{rem:SpeedupThroughBlocks}
  The crucial observation for making these computations efficient is that they all 
  can be done in sparse block form for the micro blocks in the Weyman complex. Thus 
  a module element 
  $u \in \bigoplus_{k=1}^{r_p} \check C^\bullet_{\leq e}(S)_\alpha$ 
  will be stored as 
  \[
    u = \sum_{k : u_k \neq 0} u_k, \qquad u_k \in \check C^\bullet_{\leq e}(S_R)_\alpha
  \]
  in memory. All terms possibly appearing in these sums belong to strands for different 
  $\alpha$ of one and the same direct limit of complexes of graded $S_R$-modules
  $\lim_{e \in \NN^r} C^\bullet_{\leq e}(S)$. 
  Hence creating the objects and maps of this limit only once and caching them allows for massive 
  recycling throughout the computations. This strategy is also the backbone of 
  \cite{Zach25} and explained in more detail there. For the implementation of the Weyman 
  functor we extended this already existing functionality.
  In particular, the morphisms in the \v Cech complexes, including those stemming from 
  the Gau{\ss} reductions, can be created over $\mathbf k$ rather then $R$ and then 
  be ported via base change. 

  Thus, a major part of required computations 
  for the Weyman complex depends only on the toric variety $X$ and not on the specific 
  complex of graded $S_R$-modules. In our implementation these computations are 
  cached; they could even be stored in a database and loaded upon request. 
  Moreover, for specific toric varieties, such as e.g. products of projective spaces, 
  it should be possible to produce the simplifications of the strands up to homotopy 
  in a direct, hardcoded way, instead of applying a generic Gau{\ss} reduction algorithm. 
  This should bring computation times closer to those with an already filled cache. 

  As for the part of the computation which actually depends on the complex and not 
  only on the toric variety $X$: 
  The differentials in the original complex $f^j \colon M^j \to M^{j+1}$
  can be assumed to be sparse polynomial maps over $S_R$
  and every entry 
  of the representing matrix for $f^j$ is a single homogeneous 
  polynomial $f^j_{k, l} \in (S_R)_{\alpha_{j, k} - \alpha_{j+1, l}}$.
  The individual macro and micro blocks of the horizontal maps 
  $\varphi$ induced in the \v Cech double complex 
  (\ref{eqn:CechDoubleComplex}) can be evaluated quickly from those 
  $f_{p, q}^j$. Sparsity then leads to a significant decrease in the number 
  of maps which are actually required in the computation. 
\end{remark}

\begin{example}
  \label{exp:CotantengSheafExample}
  To illustrate the benefits described in Remark \ref{rem:SpeedupThroughBlocks}, 
  let us consider the following example. 
  We fix a degree $d \in \NN$ for a family of hypersurfaces in 
  $\mathbb{P}^n_R$ over $\Spec R$ for some ring $R = \mathbb{Q}[a_1,\dots, a_r]$ 
  of parameters. On this family we then construct graded modules for certain twists 
  of the relative cotangent sheaf up to quasi-isomorphism and compute its direct image 
  on $\Spec R$.

  Suppose $S = R[x_0, \dots, x_n]$ is 
  the homogeneous coordinate ring and $F \in S_d$ a homogeneous equation of degree 
  $d$ defining the family $X \subset \mathbb{P}^n_R$ over $R$.
  We may restrict the Euler sequence 
  \[
    \begin{xy}
      \xymatrix{
        0 \ar[r]& 
        \Omega^1 \ar[r] & 
        \OO(-1)^{n+1} \ar[r]^-x & 
        \OO \ar[r] &
        0
      }
    \end{xy}
  \]
  to $X$ to obtain $\Omega^1|_X$ up to quasi isomorphism as the 
  total complex of 
  \[
    \begin{xy}
      \xymatrix{
        \OO(-d-1)^{n+1} \ar[r]^-x \ar[d]^{F \cdot \mathbf{1}}&
        \OO(-d) \ar[d]^{\cdot F}\\
        \OO(-1)^{n+1}. \ar[r]^x &
        \OO
      }
    \end{xy}
  \]
  Combining this with the short exact sequence
  \[
    \begin{xy}
      \xymatrix{
        0 \ar[r] & 
        \OO_X(-d) \ar[r]^{\mathrm{d}\, F} & 
        \Omega^1|_X \ar[r] & 
        \Omega^1_X \ar[r] &
        0 
      }
    \end{xy}
  \]
  we eventually find that the complex of graded $S$-modules given by 
  \[
    \begin{xy}
      \xymatrix{
        %0 & 
        S[-2d]\ar[r]^-{
          \begin{pmatrix}
            F & -\mathrm{d}\, F
          \end{pmatrix}
        } &
        S[-d] \oplus S[-d-1]^{n+1} \ar[r]^-{
          \begin{pmatrix}
            -d & \mathrm{d}\, F \\
            -x & F\cdot \mathbf{1}
          \end{pmatrix}
        }
        &
        % &
        % 0 \ar[l]
        S[-d] \oplus S[-1]^{n+1} \ar[r]^-{
          \begin{pmatrix}
            F \\
            x
          \end{pmatrix}
        }
        & 
        S %\ar[l] & 
      }
    \end{xy}
  \]
  has a homology module $M$ at the second term from the right hand side which 
  sheafifies to $\Omega^1_X$. All other homologies are either zero or primary to 
  the irrelevant ideal $\mathfrak m = \langle x_0, \dots, x_n\rangle$. 

  We will be interested in the direct image $R\pi_* \Omega^1_X(\lfloor 2\frac{d}{3}\rfloor)$ 
  for a family of hypersurfaces given by 
  \[
    F = \sum_{i=1}^{10} a_i \cdot x^{\nu_i}, \qquad |\nu_i| = d, \quad i = 1,\dots 10
  \]
  for ten randomly chosen monomials $x^{\nu_i}$ of degree $d$. 
  The choice of the twist has not been made because of any geometric significance 
  of this particular family of sheaves, but because it produces interesting examples 
  which illustrate the point in question here. 

  Figure \ref{fig:CotangentSheafExampleRuntimes} 
  shows runtimes for the the computation of the morphism in degree $0$ of the direct
  image $R\pi_* \Omega^1_X(\lfloor 2\frac{d}{3}\rfloor)$ for $n = 2$ and $1\leq d \leq 40$.
  We use different methods: Our implementation for 
  the Weyman functor, the Weyman functor applied, but with its cache (for all computations 
  depending only on the toric variety $X$, but not on the concrete complex) already filled, 
  and -- for comparison -- the direct Gau{\ss} reduction for the total complex 
  of the minimal truncation of the $\Ext$-analog of the total complex of (\ref{eqn:CechDoubleComplex}). 
  The actual numbers and the memory consumption are given in Table 
  \ref{tab:TimingsCotangentSheafExample} in the appendix.

  We extended the plot in Figure \ref{fig:CotangentSheafExampleRuntimes} 
  with graphs of the form $t(d) = a^{d-d_0}$, 
  fitting the different runtimes. 
  For the application of the Weyman functor, the best fit is given by 
  $a \approx 1.08832$ and $d_0 \approx 3.03249$; for direct Gau{\ss} reduction of the 
  total complex we find $a \approx 1.13837$ and $d_0 \approx -3.62572$.
  These results suggest that in both cases the runtime grows roughly exponentially 
  with the degree $d$. However, there is a significant difference 
  in the base of the exponential function so that the actual runtimes differ by orders of magnitude. 

  The ``break'' and deviations in the graph of the runtimes for the computation of 
  the Weyman complex with a filled cache is explained by julia's 
  garbage collector: The first few computations are so small
  that on average the randomly triggered 
  garbage collection does not fall within the execution time.
  From around $d = 27$ onwards this changes so that now the runtime usually comprises 
  garbage collection, too.
\end{example}

\begin{figure}[ht]
    \centering
    \includegraphics[scale=0.6, clip=true]{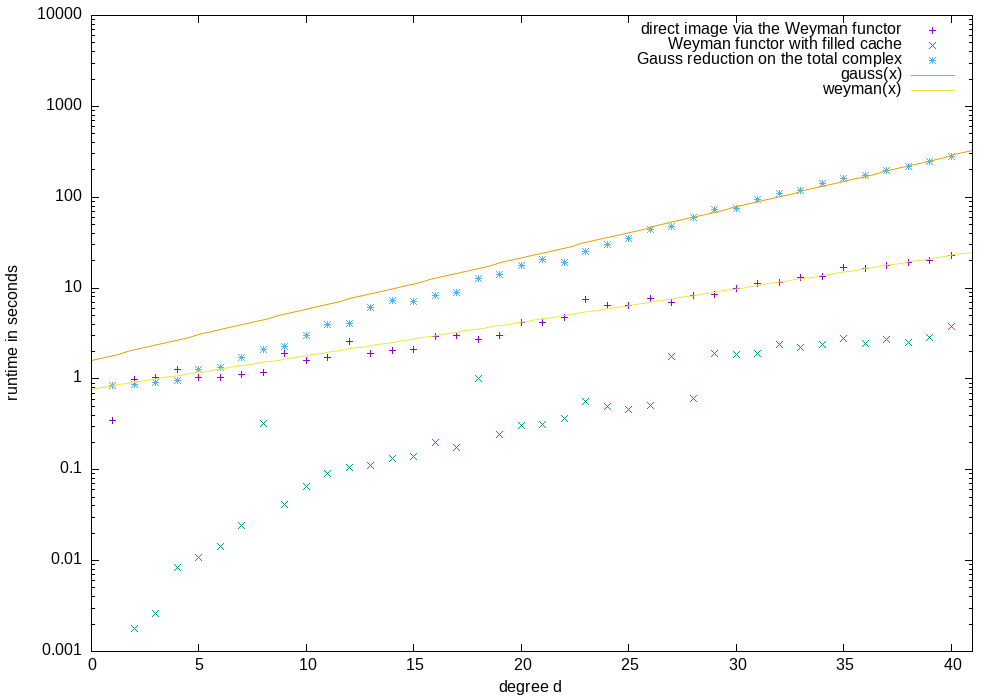}
    \caption{Runtimes for the computation of the morphism of the direct image from Example \ref{exp:CotantengSheafExample} and Table \ref{tab:TimingsCotangentSheafExample}}
    \label{fig:CotangentSheafExampleRuntimes}
\end{figure}
%While we might not be able to avoid computation with monomial bases as big as 
%in Example \ref{exp:CotantengSheafExample}
%the results presented in this paper will allow for two things:
%\begin{itemize}
%  \item Precompute the main chunks of reduction of monomial bases depending 
%    only on the toric variety $X$, but not on any specific complex of graded 
%    $S_R$-modules.
%  \item Completely avoid redundancies such as to do $n+1$ computations for 
%    the module $S[-1]$ when it appears as $S[-1]^{n+1}$ in a given complex 
%    of modules.
%\end{itemize}

\begin{remark}
  \label{rem:WellDefinednessOfDifferential}
  The construction of the differential of the Weyman complex depends on 
  numerous choices, including the exponent vector $e$ and the pivots for 
  the Gau{\ss} reduction of the individual strands. Either two different 
  choices will, in general, lead to different matrices representing the 
  differential. It is important to note that the change of basis 
  translating one matrix to the other is not only induced from the 
  direct limits 
  \[
    \check H_{\leq e}^k(M^p)_0 \to \check H_{\leq e'}^k(M^p)_0
  \]
  for $e \leq e'$ and eventual Gau{\ss} reduction to write these 
  cohomology groups as free $R$-modules. In fact, the change of basis 
  goes \textit{across macro blocks} and depends on the differentials 
  in the complex $M^\bullet$, as can be seen from the explicit 
  block lower triangular form of $U^{-1}$ above.
  Taking these base changes into account, we see that 
  the differential of the Weyman complex is indeed \textit{natural} and does 
  \textit{not} depend on the choices made in its construction.
\end{remark}

\subsection{The Weyman functor on morphisms of complexes}

Let $\theta \colon M^\bullet \to N^\bullet$ be a morphism of complexes of 
graded $S_R$-modules. We wish to describe the evaluation of the Weyman functor 
on the morphism $\theta$:
\[
  \mathbb W(\theta) \colon \mathbb W(M^\bullet) \to \mathbb W(N^\bullet).
\]
Given the explicit construction of the Weyman functor by means of linear algebra, 
the problem boils down to more or less a change of basis. 

For every exponent vector $e \in \NN^r$ the morphism $\theta$ induces a morphism 
of double complexes
\begin{equation}
  \label{eqn:MorphismOfDoubleComplexes}
  \begin{xy}
    \xymatrix{
      % zeroth row
      & \vdots & \vdots & \vdots & \vdots & \\
      % first row
      \cdots \ar[rr]^<<<<<<\psi & &
      \check C^1_{\leq e}(N^{p-1})_0 \ar[u]^{\check \partial} \ar[rr]^<<<<<<\psi & 
      &
      \check C^1_{\leq e}(M^{p})_0 \ar[u]^{\check \partial} \ar[r]^<<<<<<\psi & 
      \cdots \\
      % second row
      \cdots \ar[r]^-\varphi &
      \check C^1_{\leq e}(M^{p-1})_0 \ar[uu]^<<<<<<{\check \partial} \ar[rr]^<<<<<<<\varphi 
      \ar[ur]^\theta & 
      &
      \check C^1_{\leq e}(M^{p})_0 \ar[uu]^<<<<<<{\check \partial} \ar[rr]^<<<<<<<\varphi 
      \ar[ur]^\theta & 
      &
      \cdots
      \\
      % third row
      \cdots \ar[rr]^<<<<<<\psi & &
      \check C^0_{\leq e}(N^{p-1})_0 \ar[uu]^<<<<<<<{\check \partial} \ar[rr]^<<<<<<\psi & 
      &
      \check C^0_{\leq e}(M^{p})_0 \ar[uu]^<<<<<<<{\check \partial} \ar[r]^<<<<<<\psi & 
      \cdots \\
      % fourth row
      \cdots \ar[r]^-\varphi &
      \check C^0_{\leq e}(M^{p-1})_0 \ar[uu]^<<<<<<{\check \partial} \ar[rr]^<<<<<<<\varphi 
      \ar[ur]^\theta & 
      &
      \check C^0_{\leq e}(M^{p})_0 \ar[uu]^<<<<<<{\check \partial} \ar[rr]^<<<<<<<\varphi 
      \ar[ur]^\theta & 
      &
      \cdots
      \\
    }
  \end{xy}
\end{equation}
which, in turn, furnishes a morphism of total complexes. These form a natural direct 
limit indexed by the choices of exponent vectors $e \in \NN^r$. 

We may perform Gau{\ss} reduction of both total complexes in the domain and the codomain 
as in Lemma \ref{lem:TotalComplexRaw}. This leads to a change 
of basis in every homological degree $i$ which contains the homologies of the columns 
as direct summands: 
\[
  \bigoplus_{p + q = i} \check H^q_{\leq e}(M^p)_0 
  \overset{\iota}{\hookrightarrow}
  \bigoplus_{p + q = i} \check C^q_{\leq e}(M^p)_0 
  \overset{\theta}{\longrightarrow}
  \bigoplus_{p + q = i + 1} \check C^q_{\leq e}(M^p)_0 
  \overset{\rho}{\longrightarrow}
  \bigoplus_{p + q = i + 1} \check H^q_{\leq e}(M^p)_0.
% \bigoplus_{p + q = i} \bigoplus_{k = 1}^{r_p} \check H^q_{\leq e}(S[-\alpha_{p, k}])_0 
% \hookrightarrow
% \bigoplus_{p + q = i} \bigoplus_{k = 1}^{r_p} \check C^q_{\leq e}(S[-\alpha_{p, k}])_0 
\]
The morphism $\mathbb W(\theta)$ is then given by the composition of these maps, 
together with their natural identifications in cohomology in its domain and codomain. 

Spelled out on an element 
$u \in R^{q_0}\pi_* \OO(-\alpha_{p_0, k}) \subset R^{q_0}\pi_* \widetilde M^{p_0}$ this means 
that we first take it to the element 
$u' \in \check H^{q_0}_{\leq e}(S_R[-\alpha_{p_0, k}])_0 \subset \check C^{q_0}_{\leq e}(S_R[-\alpha_{p_0, k}])_0$ 
which is then spread out to 
\[
  \{v_{p, q}\}_{p + q = i} \in 
  \bigoplus_{p + q = i} \check C^{q}_{\leq e}(M^{p})_0,
  \quad v_{p, q} = 
  \begin{cases}
    0 & \textnormal{ if } q > q_0 \textnormal{ or } q < 0, \\
    u' & \textnormal{ if } q = q_0, \\
    (-1)^{i-1} \cdot h(\varphi(v_{p-1, q-1})) & \textnormal{ if } 0 \leq q < q_0.
  \end{cases}
\]
This element is then mapped via the block diagonal matrix for $\theta$ to 
\[
  \{w_{p, q}\}_{p + q = i} \in 
  \bigoplus_{p + q = i} \check C^{q}_{\leq e}(N^{p})_0,
  \quad w_{p, q} = \theta(v_{p, q}).
\]
On the codomain's side we invert the change of basis to get 
\[
  \{w'_{p, q}\}_{p + q = i} \in 
  \bigoplus_{p + q = i} \check C^{q}_{\leq e}(N^{p})_0,
  \quad w'_{p, q} = 
  \begin{cases}
    0 & \textnormal{ if } q > q_0 \textnormal{ or } q < 0, \\
    \theta(v_{p, q}) & \textnormal{ if } q = q_0, \\
    \theta(v_{p, q}) + (-1)^{i-1} \cdot h(\psi(w'_{p-1, q-1})) & \textnormal{ if } 0 \leq q < q_0.
  \end{cases}
\]
This element is then projected to the chosen cohomology models for the codomain. 

\begin{example}
  Let $I = \langle x_0, \dots, x_n \rangle \subset S = \QQ[x_0, \dots, x_n]$ 
  be the irrelevant ideal of projective space $\PP^n_{\QQ}$ and consider the complex 
  \[
    \begin{xy}
      \xymatrix{
        0 \ar[r] &
        I \ar[r] &
        S \ar[r] &
        0
      }
    \end{xy}
  \]
  This complex of graded modules sheafifies to the zero sheaf up to quasi isomorphism, so its direct image 
  on $\Spec \QQ$ should be zero. 

  Since $I$ is not a free module, our methods do not apply directly and we need to replace it by a free 
  resolution. The latter is provided by a truncated Koszul complex so that we obtain a morphism of complexes 
  \[
    \begin{xy}
      \xymatrix{
        & & & C^\bullet \colon &
        0 \ar[r] &
        S \ar[r] &
        0 \\
        K^\bullet \colon 
        0 \ar[r] &
        \bigwedge^{n+1} S[-1]^{n+1} \ar[r] & 
        \bigwedge^{n} S[-1]^{n+1} \ar[r] & 
        \cdots \ar[r] & 
        \bigwedge^{2} S[-1]^{n+1} \ar[r] & 
        \bigwedge^{1} S[-1]^{n+1} \ar[r] \ar[u]_\theta & 
        0.
      }
    \end{xy}
  \]
  As usual, the maps in $K^\bullet$ are given by contraction with $x = (x_0, \dots, x_n)$. 
  The Weyman complex for the one term complex $C^\bullet$ in the codomain is easily written down 
  as 
  \[
    \mathbb W^i(C^\bullet) = \begin{cases}
      H^0(\PP^n_\QQ, \OO) = \check H^0_{\leq (0, \dots, 0)}(S)_0 = \QQ \cdot \bigoplus_{j=0}^n \frac{1}{x_j^0} 
      %H^0(\PP^n_\QQ, \OO) = \check H^0_{\leq (1, \dots, 1)}(S)_0 = \QQ \cdot \bigoplus_{j=0}^n x_j \cdot \frac{1}{x_j} 
      & \textnormal{ if } i = 0\\
      0 & \textnormal{ otherwise. }
    \end{cases}
  \]
  The single non-trivial generator corresponds, of course, to the unit $1 \in S$ of the homogeneous coordinate ring 
  which restricts to the given expressions in the standard affine charts. 
  We have chosen $e_{\min}(0) = (0, \dots, 0)$, as this is the minimal exponent vector which captures all cohomology for 
  the module $S^1$. 

  For the truncated Koszul complex $K^\bullet$
  the cohomology of the free modules $\bigwedge^p S[-1]$ 
  all vanish except for $p = n+1$ where we again find only a single cohomology group. Then 
  \[
    \mathbb W^i(K^\bullet) = 
    \begin{cases}
      H^n(\PP^n, \OO(-n-1)) = \check H^n_{\leq (1, \dots, 1)}(S[-n-1]) = \QQ \cdot \frac{e_0\wedge \cdots \wedge e_n}{x_0\cdots x_n} & \textnormal{ if } i = 0 \\
      0 & \textnormal{ otherwise. }
    \end{cases}
  \]
  Here we denote the standard generators of the free module $S[-1]^{n+1}$ by $e_j$. 
  Note that here we need to choose $e_{\min}(-n-1) = (1,\dots, 1)$. 

  Since we expect the direct image to vanish, the induced map in the respective Weyman complexes 
  \[
    \check H^n_{\leq (1,\dots,1)}(S[-n-1]) = 
    \mathbb W^0(K^\bullet) \overset{\mathbb W(\theta)}{\longrightarrow} \mathbb W^0(C^\bullet) = 
    \check H^0_{\leq (1,\dots, 1)}(S)
  \]
  needs to be an isomorphism. 
  To see this explicitly, we track the mapping of the single generator $u = \frac{e_0 \wedge \cdots \wedge e_n}{x_0,\dots, x_n}$. 
  The contraction morphism takes it to 
  \[
    \sum_{0\leq j \leq n} (-1)^j x_j \cdot \frac{e_0 \wedge \cdots \wedge \widehat{e_j} \wedge \cdots \wedge e_n}{x_0 \cdots x_n}
    \in \check C^n_{\leq (1,\dots, 1)} \left(\bigwedge^{n} S[-1]^{n+1}\right)
  \]
  which is in the image of the \v Cech morphism with preimage 
  \[
    \bigoplus_{0\leq j \leq n} \frac{e_0 \wedge \cdots \wedge \widehat{e_j} \wedge \cdots \wedge e_n}{x_0 \cdots\widehat{x_j}\cdots x_n}
    \in \check C^{n-1}_{\leq (1,\dots, 1)} \left(\bigwedge^{n} S[-1]^{n+1}\right).
  \]
  Given that our cohomology groups are only one-dimensional, we may assume that the 
  homotopy $h$ has been chosen so that this is really the image of our element under $h$. 
  This eventually continues so that we obtain the collection of elements 
  \[
    \{v_{-k, k}\}_{k = 0}^n, \quad v_{-k, k} = \bigoplus_{0\leq i_0 < \dots < i_k \leq n} 
    \frac{e_{i_0}\wedge \cdots \wedge e_{i_k}}{x_{i_1} \cdots x_{i_k}} 
    \in \check C_{\leq (1, \dots, 1)}^k\left(\bigwedge^{k+1} S[-1]^{n+1}\right)_0
  \]
  according to the change of basis for the Weyman complex. 
  Note that for $k=n$ we find that $v_{-n, n} = u$. 
  These elements are taken to zero modules in the codomain except for $k = 1$.
  In this last case one has
  \[
    v_{-1, 1} = \bigoplus_{0\leq i \leq n} \frac{e_i}{x_i} \overset{\theta}{\mapsto} \bigoplus_{i=0}^n x_i \cdot \frac{1}{x_i} 
    \in \check C^0_{\leq (1,\dots, 1)}(S)_0.
  \]
  Canceling fractions it is now easy to see that the cohomology class of this element coincides with the one given by 
  $\bigoplus_{j=0}^n \frac{1}{x_j^0}$ under the natural map $\check C^0_{\leq (0, \dots, 0)}(S)_0 \to \check C_{\leq (1, \dots, 1)}^0(S)_0$. 
  The inverse of the induced map in cohomology is precisely the projection to the chosen standard cohomology model in the Weyman complex.
\end{example}

\section{Examples}

\subsection{Sturmfels' example}
\label{sec:SturmfelsExample}

In \cite[Example 2.1]{Sturmfels93} Sturmfels sets out to compute the 
$A$-resultant for the following support sets
\begin{equation}
  \label{eqn:SupportSetsSturmfelsExample}
  A_1 = \left\{
    (0, 0), (2, 2), (1, 3)
  \right\}, \qquad
  A_2 = \left\{
    (0, 0), (2, 0), (1, 2)
  \right\}, \qquad
  A_3 = \left\{
    (3, 0), (1, 1)
  \right\}
\end{equation}
Note that these support sets do not satisfy the criteria made in the 
introduction since the convex hull of $A_3$ is not full-dimensional. 
However, Proposition \ref{prp:BroaderApplicability} applies, so that we 
may use our machinery on this example.
A direct computation via elimination takes
34.29979 seconds with 11.589 GiB of memory required. 
The eliminant is 
\begin{eqnarray*}
  & & a_1^5 b_3^7 c_1^6 c_2 + 3 a_1^4 a_2 b_2^2 b_3^5 c_1^4 c_2^3 + 3 a_1^3 a_2^2 b_2^4 b_3^3 c_1^2 c_2^5 - 13 a_1^3 a_2 a_3 b_1^2 b_2 b_3^4 c_1^5 c_2^2 - 7 a_1^3 a_3^2 b_1 b_2^3 b_3^3 c_1^4 c_2^3\\
  & +& 6 a_1^2 a_2^3 b_1^3 b_2 b_3^3 c_1^4 c_2^3 + a_1^2 a_2^3 b_2^6 b_3 c_2^7 - a_1^2 a_2^2 a_3 b_1^2 b_2^3 b_3^2 c_1^3 c_2^4 + 5 a_1^2 a_2 a_3^2 b_1^4 b_3^3 c_1^6 c_2 - a_1^2 a_2 a_3^2 b_1 b_2^5 b_3 c_1^2 c_2^5 \\
  & +& 14 a_1^2 a_3^3 b_1^3 b_2^2 b_3^2 c_1^5 c_2^2 + a_1^2 a_3^3 b_2^7 c_1 c_2^6 - 2 a_1 a_2^4 b_1^3 b_2^3 b_3 c_1^2 c_2^5 - 5 a_1 a_2^3 a_3 b_1^5 b_3^2 c_1^5 c_2^2 + 2 a_1 a_2^2 a_3^2 b_1^4 b_2^2 b_3 c_1^4 c_2^3 \\
  &-& 2 a_1 a_2 a_3^3 b_1^3 b_2^4 c_1^3 c_2^4 - 7 a_1 a_3^4 b_1^5 b_2 b_3 c_1^6 c_2 + a_2^5 b_1^6 b_3 c_1^4 c_2^3 + a_2^2 a_3^3 b_1^6 b_2 c_1^5 c_2^2 + a_3^5 b_1^7 c_1^7
\end{eqnarray*}
where we write $a_1, a_2, a_3$ for the coefficients associated to the 
vertices of $A_1$, $b_1, b_2, b_3$ for those of $A_2$, and $c_1, c_2$ 
for $A_3$. The multiplicity of the $A$-resultant is equal to one in this example 
and thus the eliminant coincides with $\Delta_{A_1, A_2, A_3}$.
As usual, we denote the polynomial ring in these variables over $\QQ$ by $R$. 

The toric variety $X$ for the Minkowski sum of the convex hulls for these support sets 
has the rays 
\[
  \begin{pmatrix}
    -1\\ -2
  \end{pmatrix},
  \begin{pmatrix}
    -2\\ -1
  \end{pmatrix},
  \begin{pmatrix}
    -1\\ -1
  \end{pmatrix},
  \begin{pmatrix}
    2\\ -1
  \end{pmatrix},
  \begin{pmatrix}
    3\\ -1
  \end{pmatrix},
  \begin{pmatrix}
    0\\ 1
  \end{pmatrix},
  \begin{pmatrix}
    -1\\ 1
  \end{pmatrix},
  \begin{pmatrix}
    1\\ 2
  \end{pmatrix}.
\]
The corresponding variables of the Cox ring are graded by 
\begin{eqnarray*}
  & & x_1 \mapsto [1, 0, 0, 0, 0, 0],\quad
 x_2 \mapsto [0, 1, 0, 0, 0, 0],\quad
 x_3 \mapsto [0, 0, 1, 0, 0, 0],\quad
 x_4 \mapsto [0, 0, 0, 1, 0, 0],\\
  & & 
 x_5 \mapsto [0, 0, 0, 0, 1, 0],\quad
 x_6 \mapsto [0, 0, 1, 2, 2, 3],\quad
 x_7 \mapsto [-1, 1, 0, -1, -1, -2],\quad
 x_8 \mapsto [1, -1, 2, 2, 0, 1]
\end{eqnarray*}
where the grading takes values in the group $\ZZ^6$ of toric divisor 
classes on $X$. For the twist by $(1,\dots, 1)$ the Weyman complex 
has only two terms, both of which are of rank 15.
These are distributed according to the first page of the spectral sequence 
(\ref{eqn:TheSpectralSequence}) 
which reads 
\[
  \begin{array}{c|cccc}
    2 & R^{15} & R^{12} & 0 & 0 \\
    1 & 0 & 0 & R^2 & 0 \\
    0 & 0 & 0 & 0 & R^1\\
    \hline
    q/p & -3 & -2 & -1 & 0.
  \end{array}
\]
The representing matrix for the morphism between them takes the following form: 
\[
  \begin{pmatrix}
                     0&             0&             0&      0&      0&      0&      0&      0&      0&  -b_1&       0&       0&       0&  a_1&      0\\
                     0&             0&             0&      0&      0&      0&      0&      0&      0&       0&  -b_1&       0&       0&      0&  a_1\\
                     0&             0&             0&  c_1&      0&      0&      0&      0&      0&  -b_2&       0&       0&       0&      0&      0\\
                     0&   a_1 b_3&             0&  c_2&      0&      0&      0&      0&      0&       0&       0&  -b_1&       0&      0&      0\\
                     0&             0&             0&      0&  c_1&      0&      0&      0&      0&       0&  -b_2&       0&       0&      0&      0\\
                     0&  -a_2 b_1&   a_1 b_3&      0&  c_2&      0&      0&      0&      0&       0&       0&       0&  -b_1&      0&      0\\
    -a_1 b_3 c_1&             0&  -a_2 b_1&      0&      0&  c_1&      0&      0&      0&       0&       0&  -b_2&       0&      0&      0\\
                     0&  -a_3 b_1&             0&      0&      0&  c_2&      0&      0&      0&       0&       0&       0&       0&      0&      0\\
                     0&  -a_2 b_2&             0&      0&      0&      0&  c_1&      0&      0&       0&       0&       0&  -b_2&      0&      0\\
    -a_2 b_1 c_2&             0&  -a_3 b_1&      0&      0&      0&  c_2&      0&      0&  -b_3&       0&       0&       0&      0&      0\\
                     0&             0&  -a_2 b_2&      0&      0&      0&      0&  c_1&      0&       0&       0&       0&       0&      0&      0\\
     a_3 b_1 c_1&  -a_3 b_2&             0&      0&      0&      0&      0&  c_2&      0&       0&  -b_3&       0&       0&  a_2&      0\\
    -a_2 b_2 c_2&             0&  -a_3 b_2&      0&      0&      0&      0&      0&  c_1&       0&       0&       0&       0&      0&  a_2\\
                     0&             0&             0&      0&      0&      0&      0&      0&  c_2&       0&       0&  -b_3&       0&  a_3&      0\\
    -a_3 b_2 c_2&             0&             0&      0&      0&      0&      0&      0&      0&       0&       0&       0&  -b_3&      0&  a_3\\
  \end{pmatrix}.
\]
The determinant of this matrix does coincide with the resultant up to sign. One can observe how the degree of the entries 
of a column coincides with the number of steps in the staircase (\ref{eqn:StaircaseLifting}) which need to be taken 
to realize the corresponding block map; cf. \cite[Chapter 5, Theorem 5.5 (a)]{GelfandKapranovZelevinsky94}.
It takes 20.683419 seconds to arrive at this matrix with only 354.649 MiB of memory allocated. 
If we presume the cache for the Weyman functor to already be filled, the computation is carried out in less then 50 milliseconds. 

In \cite{Sturmfels93} Sturmfels uses a stable twist in order to compute 
the determinant of the complex of global sections. In our setup this amounts to twisting with 
the toric divisor class\footnote{Note that this is not a Cartier divisor.} 
of degree $[4, 7, 16, 12, 3, 2] \in \ZZ^6$. 
We obtain a three-term complex of global sections of the form 
\[
  \begin{xy}
    \xymatrix{
      0 \ar[r]& 
      R^{23} \ar[r] &
      R^{27} \ar[r] &
      R^{4} \ar[r] &
      0
    }
  \end{xy}
\]
but it seems impossible to come up with a stable twist to get only a two-term complex. 
The matrices for the maps are too big to be reproduced here; their entries have polynomial 
degree at most 1. 

The computation of the complex of global sections takes 15.552580 seconds with 167.628 MiB of allocated memory.
More than 95\% of this time is spent on solving linear programs on polytopes for the computation 
of $e_{\min}$ and collecting lattice points within such polytopes for monomial bases of strands of graded modules. 
With a filled cache for the Weyman functor the complex above is provided within 36 milliseconds. 
Computing the determinant of that complex takes another 8 milliseconds. 

\subsection{Parametric images of rational curves}
\label{sec:ParametricImages}

Suppose we are given an algebraic family of morphisms 
\[
  F \colon \PP^1 \times \CC \to \PP^2, \quad (x, t) \mapsto f_t(x) \in \PP^2
\]
i.e. $(f_t)_i(x) = \sum_{|\nu| = d} a_{i,\nu}(t)\cdot x^\nu, \quad i = 0, 1, 2$ for some fixed 
degree $d \in \NN$ and coefficients $a_{i, \nu}\in \CC[t]$. We obtain a family of images 
$Y \subset \PP^2 \times \CC \to \CC$ such that $Y_t = Y \cap \PP^2 \times \{t\} = f_t(\PP^1)$.
The challenge is to describe the implicit equation for $Y$. 

For this benchmark example, we choose the coefficients $a_{i, \nu}$ to depend affine linearly 
on $t$ with fixed, random coefficients. Our machine also works for generic coefficients, but we 
restrict ourselves to $1$-parameter families, so as to provide a setup in which elimination methods 
have a chance to finish in reasonable time. 

We build the ideal for the graph generated by 
\[
  (f_t)_0(x_0, x_1) \cdot u - (f_t)_1(x_0, x_1), \quad
  (f_t)_0(x_0, x_1) \cdot v - (f_t)_2(x_0, x_1), 
\]
where $(x_0, x_1)$ are the homogeneous coordinates for $\PP^1$ and $(u, v)$ are affine coordinates 
of the codomain $\PP^2$. To compute the equation for the image via elimination, we dehomogenize 
by setting $x_0 = 1$ and then eliminate $x_1$ from the ideal. For comparison, this is also done 
in positive characteristic over the field $\mathbb F_{101}$. 
When applying our machinery for $A$-resultants, we consider the above system over the variety 
$\PP^1\times \CC \times \CC^2$ with $X = \PP^1$ 
the toric variety arising from the ``triangular'' support 
sets $A_i = \{(0), \dots, (d)\}$. 

\begin{table}
  \begin{center}
    \begin{tabular}{|l|l|l|l|l|l|l|l|l|}
      \hline
      Deg. & \multicolumn{4}{l}{stable twist} & 
      \multicolumn{4}{l}{optimized twist}\\
      \hline
      $d$ & \multicolumn{2}{l}{direct image} & \multicolumn{2}{l}{determinant} & 
      \multicolumn{2}{l}{direct image} & \multicolumn{2}{l}{determinant} \\
      \hline
      1 & 
      0.257670 s & 4.785 MiB &
      %[0, 3, 4, 1, 0]
      0.001632 s & 48.234 KiB &
      %optimal twist
      0.241543 s & 3.335 MiB & 
      %[0, 1, 1, 0, 0]
      0.000125 s & 8.328 KiB \\ 
      2 & 
      0.219085 s & 3.831 MiB &
      %[0, 5, 6, 1, 0]
      0.002397 s & 117.418 KiB & 
      %optimal twist
      0.331929 s & 3.007 MiB & 
      %[0, 2, 2, 0, 0]
      0.001529 s & 19.672 KiB \\ 
      3 & 
      0.217710 s & 4.473 MiB &
      %[0, 7, 8, 1, 0]
      0.002790 s & 391.711 KiB & 
      %optimal twist
      0.323865 s & 3.665 MiB & 
      %[0, 3, 3, 0, 0]
      0.001819 s & 51.633 KiB \\ 
      4 & 
      0.215691 s & 5.185 MiB &
      %[0, 9, 10, 1, 0]
      0.006528 s & 1.040 MiB & 
      %optimal twist
      0.318101 s & 4.668 MiB & 
      %[0, 5, 5, 0, 0]
      0.002821 s & 276.250 KiB \\ 
      5 & 
      0.218553 s & 5.951 MiB &
      %[0, 11, 12, 1, 0]
      0.022298 s & 3.142 MiB & 
      %optimal twist
      0.327058 s & 5.571 MiB & 
      %[0, 6, 6, 0, 0]
      0.007956 s & 673.281 KiB \\ 
      6 & 
      0.217042 s & 6.774 MiB &
      %[0, 13, 14, 1, 0]
      0.152825 s & 7.375 MiB & 
      %optimal twist
      0.323507 s & 6.570 MiB & 
      %[0, 7, 7, 0, 0]
      0.026254 s & 1.434 MiB \\ 
      7 & 
      0.222800 s & 7.672 MiB &
      %[0, 15, 16, 1, 0]
      0.542230 s & 15.826 MiB & 
      %optimal twist
      0.321827 s & 8.220 MiB & 
      %[0, 9, 9, 0, 0]
      0.169449 s & 4.852 MiB \\ 
      8 & 
      0.223987 s & 8.792 MiB &
      %[0, 17, 18, 1, 0]
      1.536745 s & 30.783 MiB & 
      %optimal twist
      0.325753 s & 9.468 MiB & 
      %[0, 10, 10, 0, 0]
      0.399394 s & 8.495 MiB \\ 
      9 & 
      0.223444 s & 9.880 MiB &
      %[0, 19, 20, 1, 0]
      3.555785 s & 57.665 MiB & 
      %optimal twist
      0.335605 s & 10.943 MiB & 
      %[0, 11, 11, 0, 0]
      0.936593 s & 15.404 MiB \\ 
      10 & 
      0.227206 s & 11.002 MiB &
      %[0, 21, 22, 1, 0]
      7.591414 s & 207.069 MiB & 
      %optimal twist
      0.333723 s & 12.779 MiB & 
      %[0, 13, 13, 0, 0]
      3.363516 s & 37.346 MiB \\ 
      11 & 
      0.236053 s & 12.193 MiB &
      %[0, 23, 24, 1, 0]
      13.167125 s & 815.313 MiB & 
      %optimal twist
      0.334076 s & 14.486 MiB & 
      %[0, 14, 14, 0, 0]
      6.450076 s & 137.334 MiB \\ 
      12 & 
      0.233032 s & 13.355 MiB &
      %[0, 25, 26, 1, 0]
      22.796134 s & 1.784 GiB & 
      %optimal twist
      0.340848 s & 16.242 MiB & 
      %[0, 15, 15, 0, 0]
      9.609917 s & 318.037 MiB \\ 
      13 & 
      0.228131 s & 14.709 MiB &
      %[0, 27, 28, 1, 0]
      42.269698 s & 3.539 GiB & 
      %optimal twist
      0.357773 s & 18.747 MiB & 
      %[0, 17, 17, 0, 0]
      21.653534 s & 1.193 GiB \\ 
      14 & 
      0.237664 s & 16.059 MiB &
      %[0, 29, 30, 1, 0]
      72.839881 s & 5.543 GiB & 
      %optimal twist
      0.345428 s & 20.900 MiB & 
      %[0, 18, 18, 0, 0]
      36.186837 s & 2.204 GiB \\ 
      15 & 
      0.235851 s & 17.465 MiB &
      %[0, 31, 32, 1, 0]
      117.974539 s & 9.468 GiB & 
      %optimal twist
      0.351999 s & 23.178 MiB & 
      %[0, 19, 19, 0, 0]
      62.510157 s & 3.780 GiB \\ 
      \hline
    \end{tabular}
  \end{center}
  \caption{Timings and memory consumption for the computation of the parametric image of a random family of rational curves}
  \label{tab:TimingsRationalCurve}
\end{table}

\begin{table}
  \begin{center}
    \begin{tabular}{|l|l|l|l|l|l|l|l|l|}
      \hline
      Deg. & \multicolumn{2}{l}{char. $0$} & 
      \multicolumn{2}{l}{char. $101$}\\
      \hline
      1 & 0.005592 s &  18.758 KiB &
      0.005427 s &  19.484 KiB \\
      2 & 0.001284 s &  101.648 KiB &
      0.005597 s &  22.164 KiB \\
      3 & 0.013487 s &  17.864 MiB &
      0.002422 s &  22.930 KiB \\
      4 & 2.974198 s &  5.978 GiB &
      0.021399 s &  27.523 KiB \\
      5 & -- & -- &
      0.272273 s &  31.352 KiB \\
      6 & -- & -- &
      6.796502 s &  33.266 KiB \\
      7 & -- & -- &
      -- & -- \\
      \hline
    \end{tabular}
  \end{center}
  \caption{Timings and memory consumption for the computation of the parametric image of a random family of rational curves via elimination in characteritic $0$ and $101$}
  \label{tab:TimingsRationalCurveElimination}
\end{table}

In Table \ref{tab:TimingsRationalCurveElimination} we list the timings for computation 
of the implicit equation of the family via elimination. Fields without entry indicate that 
the computation was aborted after several minutes without result. While we obtain 
rather quick results for degrees $\leq 4$, the complexity 
seems to explode with increasing degree; even though we eliminate but one variable out 
of only four in total.

We see from Table \ref{tab:TimingsRationalCurve} that the computation of the 
direct image complex remains relatively cheap with increasing degree $d$, whereas 
the computation of the determinant of said complex seems to have exponentially growing 
complexity. While the direct image for a stably twisted complex is cheaper, the size 
of the matrices to be considered for the determinant is bigger compared to a more optimal 
twisting. This pays off in the subsequent computation of determinants. Indeed, for $d = 15$ 
the ranks of the non-zero modules in the direct image for the stably twisted complex are 
$31$, $32$, and $1$, while a twist of the Koszul complex 
by $\lfloor \frac{2d}{3}\rfloor$ gives $19$ and $19$. 

\subsection{A resultant with higher multiplicity in three variables}

The support sets chosen for this example are 
\begin{eqnarray*}
  A_0 = \{(0, 0, 0), (0, 2, 4), (-2, 5, 8)\},& \quad& 
  A_1 = \{(-2, 4, 6), (1, 0, 1), (4, -4, -4)\}, \\
  A_2 = \{(3, -3, -3), (0, 1, 2)\}, &\quad& 
  A_3 = \{(0, 0, 0), (2, -4, -4)\}
\end{eqnarray*}
in $\ZZ^3$ and the multiplicity (see Section \ref{sec:BroadeningScope}) of the 
$A$-resultant over the eliminant is $14$. Computing the 
direct image for the application of Equation (\ref{eqn:DeterminantOfWeymanComplex}) 
takes us 1824.227613 seconds with 1014.679 GiB of allocated memory. This reduces to 
0.169954 seconds and 32.308 MiB of allocated memory with the cache for 
the Weyman functor already filled. Computation of the determinant of that 
complex then takes only another 0.050967 seconds and 7.287 MiB of allocations 
and we arrive at a polynomial 
$\Delta \in R = \QQ[a_{i, \nu} : i = 0, \dots, 3, \, \nu \in A_i]$ with 120 terms. 
The $14$-th root of $\Delta$ exists and is equal to 
\[
  h = a_{1, (-2, 4, 6)}\cdot a_{2, (3, -3, -3)}^2 - a_{1, (1, 0, 1)} \cdot a_{2, (3, -3, -3)} \cdot a_{2, (0, 1, 2)} + a_{1, (4, -4, -4)} \cdot a_{2, (0, 1, 2)}^2.
  %a_2_1\cdot a_3_1^2 - a_2_2 \cdot a_3_1 \cdot a_3_2 + a_2_3 \cdat a_3_2^2
\]
Interestingly this does not involve the coefficients of neither $A_0$, nor $A_3$. 

The toric variety associated to this problem has seven rays. 
On the first page of the spectral sequence (\ref{eqn:TheSpectralSequence})
one finds the following free modules 
\[
  \begin{array}{c|ccccc}
    3 & R^{19}& R^{20}&  R^{1}& 0& 0\\
    2 & 0& R^{21}& R^{21}& R^{1}& 0\\
    1 & 0&  0&  R^{2}& R^{2}& 0\\
    0 & 0&  0&  0& 0& R^{1}\\
    \hline
    q/p & -4 & -3 & -2 & -1 & 0.
  \end{array}
\]
This setting provides ample opportunity for non-trivial ``knight-move'' maps in 
the spectral sequence and the Weyman complex. Indeed, we find that the block maps 
\[
  \phi_{-2, 1}^2 \colon R^{2} \to R^1, \textnormal{ and }
  \phi_{-3, 2}^2 \colon R^{21} \to R^2
\]
are non-zero while 
\[
  \phi_{-3, 3}^2 \colon R^{20} \to R^1, \quad 
  \phi_{-3, 2}^3 \colon R^{21} \to R^1, \textnormal{ and } 
  \phi_{-4, 3}^r \textnormal{ for } r > 0
\]
all vanish.

\subsection{A random family of support sets}
To provide one more example where the toric variety $X$ is not a projective space, 
consider the following family of support sets for $k > 0$.
\begin{equation}
  \label{eqn:ScalableSupportSets}
  A_1 = \{(0, 0), (0, 3\cdot k), (3, 2\cdot k)\}, \quad
  A_2 = \{
    (1, 1), (2, 2), (k, 2\cdot k)
  \}, \quad
  A_3 = \{
    (-1, 5), (-2, 2), (3, 0)
  \}
% A_1 = 
% \begin{bmatrix}
%   0 &0\\ 0 &3\cdot k\\ 3 &2\cdot k
% \end{bmatrix}, \quad
% A_2 = \begin{bmatrix}
%   1 &1\\ 2 &2\\ k &2\cdot k
% \end{bmatrix}, \quad
% A_3 = \begin{bmatrix}
%   -1 &5\\ -2 &2\\ 3 &0
% \end{bmatrix}
\end{equation}
For every $k$ we twist the Koszul complex on the toric variety $X$ by 
$-2\cdot K_X$, a multiple of the canonical divisor. The timings for the computation of 
the direct images with this twist can be found in Table \ref{tab:TimingsScalableSupportSets}.
For $k=8$ the toric variety has $9$ rays and the ranks of the non-zero modules in 
the direct image are $8$, $364$, and $356$. Again, the matrices are too big to compute 
the full determinant of such complexes; but the description of the discriminant 
as the degeneracy locus of such a complex is often preferable over a single equation, 
anyway. 
\begin{table}
  \begin{center}
    \begin{tabular}{|l||l|l|l|l|l|l|l|l|}
      \hline
      $k$ & 1 & 2 & 3 & 4 & 5 & 6 & 7 & 8 \\
      \hline 
      time ($s$) & 
      31.966781 &
      34.338416 &
      77.535535 &
      86.226101 &
      99.094804 &
      124.743093 &
      189.448820 &
      257.901875 \\

      memory & 
      1.203 GiB &
      2.941 GiB &
      7.910 GiB &
      9.796 GiB &
      12.000 GiB &
      31.474 GiB &
      31.116 GiB &
      57.223 GiB \\

      \hline
    \end{tabular}
  \end{center}
  \caption{Timings and memory consumption for the computation of the direct image for the support sets in (\ref{eqn:ScalableSupportSets})}
  \label{tab:TimingsScalableSupportSets}
\end{table}

\printbibliography
\section{Appendix}
%\begin{table}[h]
  \noindent\begin{minipage}{\textwidth}
  \begin{center}
    \begin{tabular}{|l|l|l|l|l|l|l|l|}
      \hline
      Deg. & \multicolumn{2}{l}{Weyman complex} & 
      \multicolumn{2}{l}{W.C. with filled cache}& 
      global 
      & \multicolumn{2}{l}{total complex} \\
      \hline
      $d$ & time ($s$) & memory & time ($s$) & memory &
      $e_{\min}$ 
      & time ($s$) & memory \\
      \hline
      1     & 0.349505  & 2.183 MiB       &  0.000114    & 128.391 KiB  &   1     &  0.847143    & 27.873 MiB \\
      2     & 0.978917  & 9.686 MiB       &   0.001795   & 1.278 MiB    &     2   &   0.867675   & 35.968 MiB \\
      3     & 1.025479  & 16.348 MiB      &  0.002640    & 1.956 MiB    &    2    &   0.912313   &  50.881 MiB \\
      4     & 1.269116  & 31.889 MiB      &  0.008439    & 4.876 MiB    &   4     &  0.963399    & 97.582 MiB \\
      5     & 1.047233  & 52.978 MiB      &    0.010910  & 9.446 MiB    &   6     &  1.264132    & 179.033 MiB \\
      6     & 1.043226  & 70.532 MiB      &    0.014172  & 12.193 MiB   &   6     &  1.334998    & 245.344 MiB \\
      7     & 1.107916  & 110.535 MiB     &   0.024167   & 20.515 MiB   &   8     &  1.731219    & 388.416 MiB \\
      8     & 1.167850  & 170.018 MiB     &   0.321515   & 35.414 MiB   &   10    &  2.100965    & 642.834 MiB \\
      9     & 1.915721  & 185.126 MiB     &   0.041010   & 33.428 MiB   &   10    &  2.265868    & 680.995 MiB \\
      10    & 1.592025  & 244.932 MiB     &   0.065679   & 46.465 MiB   &     12  &    3.000391  & 893.446 MiB \\
      11    & 1.741202  & 331.905 MiB     &   0.090125   & 62.450 MiB   &     14  &    3.995461  & 1.127 GiB \\
      12    & 2.593362  & 380.803 MiB     &   0.105257   & 70.451 MiB   &     14  &    4.065723  & 1.397 GiB \\
      13    & 1.915596  & 483.991 MiB     &  0.110633    & 92.036 MiB   &     16  &    6.082456  & 1.808 GiB \\
      14    & 2.044550  & 585.171 MiB     &  0.133449    & 110.303 MiB  &     18  &    7.288717  & 2.145 GiB \\
      15    & 2.107704  & 634.886 MiB     &  0.138961    & 114.216 MiB  &     18  &    7.129719  & 2.278 GiB \\
      16    & 2.932485  & 813.741 MiB     &  0.198831    & 141.760 MiB  &     20  &    8.341343  & 2.677 GiB \\
      17    & 3.036552  & 896.070 MiB     &  0.175430    & 148.308 MiB  &     22  &    8.835953  & 2.686 GiB \\
      18    & 2.745122  & 1.031 GiB       &  1.019911    & 179.156 MiB  &     22   &    12.586285 &  3.475 GiB \\
      19    & 2.983508  & 1.213 GiB       &  0.242226    & 206.849 MiB  &     24   &    14.197863 &  3.978 GiB \\
      20    & 4.185056  & 1.485 GiB       &  0.308577    & 253.038 MiB  &     26   &    17.783226 &  4.698 GiB \\
      21    & 4.178524  & 1.583 GiB       &  0.316916    & 256.112 MiB  &     26   &    20.773759 &  4.997 GiB \\
      22    & 4.794206  & 1.845 GiB       &  0.363771    & 289.298 MiB  &     28   &    19.302929 &  5.195 GiB \\
      23    & 7.529512  & 2.122 GiB       &  0.560900    & 318.048 MiB  &     30   &    25.241289 &  5.822 GiB \\
      24    & 6.441480  & 2.317 GiB       &  0.493041    & 350.728 MiB  &     30   &    30.402498 &  6.466 GiB \\
      25    & 6.414402  & 2.586 GiB       &  0.460874    & 374.417 MiB  &     32   &    35.474589 &  6.936 GiB \\
      26    & 7.615630  & 3.038 GiB       &  0.513239    & 429.768 MiB  &     34   &    44.055398 &  7.746 GiB \\
      27    & 6.950785  & 3.233 GiB       &  1.766022    & 438.364 MiB  &     34   &    48.122962 &  7.607 GiB \\
      28    & 8.184260  & 3.664 GiB       &   0.613829   & 499.062 MiB  &    36    &   59.547643  & 8.632 GiB \\
      29    & 8.525542  & 4.204 GiB       &   1.882132   & 572.907 MiB  &    38   &   73.467864  & 10.052 GiB \\
      30    & 9.861769  & 4.516 GiB       &   1.853813   & 586.525 MiB  &    38   &   75.614470  & 10.041 GiB \\
      31    & 11.139724 &  5.175 GiB      &  1.882798    & 650.793 MiB  &    40   &   94.480961  & 11.037 GiB \\
      32    & 11.441082 &  5.709 GiB      &  2.365171    & 713.414 MiB  &    42   &   110.174999 &  11.716 GiB \\
      33    & 12.987692 &  6.281 GiB      &  2.202816    & 771.769 MiB  &    42   &   119.582474 &  13.299 GiB \\
      34    & 13.529525 &  7.008 GiB      &  2.413210    & 821.759 MiB  &    44   &   142.311629 &  13.677 GiB \\
      35    & 16.650439 &  7.839 GiB      &  2.804185    & 913.502 MiB  &    46   &   160.843573 &  15.075 GiB \\
      36    & 16.500480 &  8.230 GiB      &  2.472642    & 902.511 MiB  &    46   &   173.413111 &  14.439 GiB \\
      37    & 17.738655 &  9.073 GiB      &  2.734722    & 1015.709 MiB &    48   &   195.035378 &  15.943 GiB \\
      38    & 19.001251 &  9.882 GiB      &  2.515698    & 1011.516 MiB &    50   &   216.408133 &  15.113 GiB \\
      39    & 19.978116 &  10.766 GiB     & 2.879243     & 1.111 GiB    &    50   &   245.959172 &  17.800 GiB \\
      40    & 22.775290 &  12.001 GiB     & 3.752681     & 1.250 GiB    &    52   &   277.660898 &  18.741 GiB \\
      \hline
    \end{tabular}
  \end{center}
    \captionof{table}{Timings and memory consumption for the computation of the morphism in degree zero 
    of the direct images 
    from Example \ref{exp:CotantengSheafExample} for $n = 2$}
  \label{tab:TimingsCotangentSheafExample}
\end{minipage}
%\end{table}

\end{document}